\documentclass[onefignum,onetabnum]{siamart220329}

\usepackage{amsfonts,bm,bbm}
\usepackage{algorithmic}
\usepackage{enumerate}
\usepackage{appendix}
\usepackage{comment}
\usepackage{graphicx}
\usepackage[caption=false]{subfig}
\usepackage[caption=false]{subfig}

\newsiamthm{assumption}{Assumption}
\newsiamremark{condition}{Condition}
\newsiamremark{remark}{Remark}

\newcommand{\bbR}{\mathbb{R}}

\newcommand{\dd}{\textnormal{d}}
\newcommand{\vt}[1]{\left\vert #1 \right\vert}
\newcommand{\vvt}[1]{\left\Vert #1 \right\Vert}
\newcommand{\vta}[1]{\vert #1 \vert}
\newcommand{\vvta}[1]{\Vert #1 \Vert}
\newcommand{\lr}[1]{\left( #1 \right)}
\newcommand{\lrr}[1]{\left[ #1 \right]}
\newcommand{\lrrr}[1]{\left\{ #1 \right\}}

\newcommand{\ie}{\textit{i}.\textit{e}., }

\newcommand{\tn}[1]{\textnormal{#1}}

\usepackage[colorinlistoftodos]{todonotes}
\usepackage[square,numbers]{natbib}
\newcommand\newtext[1]{#1}

\title{New Methods for Parametric Optimization via Differential Equations}
\author{Heyuan Liu\thanks{Department of Industrial Engineering and Operations Research, University of California, Berkeley, CA 94720
({mailto:  heyuan\_liu@berkeley.edu}).}
\and Paul Grigas\thanks{Department of Industrial Engineering and Operations Research, University of California, Berkeley, CA 94720
({mailto:  pgrigas@berkeley.edu}).  This author's research is supported, in part, by NSF AI Institute for Advances in Optimization Award 2112533.}}
\date{} %

\begin{document}
\maketitle

\begin{abstract}
    We develop and analyze several different second-order algorithms for computing a near-optimal solution path of a convex parametric optimization problem with smooth Hessian. Our algorithms are inspired by a differential equation perspective on the parametric solution path and do not rely on the specific structure of the objective function. 
    We present computational guarantees that bound the oracle complexity to achieve a near-optimal solution path under different sets of smoothness assumptions. 
    Under the assumptions, the results are an improvement over the best-known results of the grid search methods. 
    We also develop second-order conjugate gradient variants that avoid exact computations of Hessians and solving of linear equations. 
    We present computational results that demonstrate the effectiveness of our methods over grid search methods on both real and synthetic datasets. 
    On large-scale problems, we demonstrate significant speedups of the second-order conjugate variants as compared to the standard versions of our methods.
\end{abstract}

\section{Introduction}

    In many applications of interest, it is necessary to solve not just a single optimization problem but also an entire collection of related problems. In these settings, some or all of the objects involved in defining the objective function or constraints of an optimization problem depend on one or more parameters, and we would like to solve the problem as a function of these parameters.
    Generally, a \emph{parametric optimization problem} can be written as: 
    \begin{equation}\label{eq:param-opt}
        P(\lambda): \quad \min_{x \in S(\lambda)} F(x, \lambda), 
    \end{equation}
    where $\lambda$ belongs to the set of interest $\Lambda \subseteq \bbR^m$, and the feasible sets satisfy $S(\lambda) \subseteq \bbR^p$. 
    There are many problems of interest that are formulated as parametric optimization problems of the form \cref{eq:param-opt}. Subsequently, as indicated by \citet{guddat1990parametric}, there are several strong motivations to design algorithms for $\cref{eq:param-opt}$, including but not limited to: {\em (i)} the need to solve problems arising in application areas like regularized regression with cross-validation (see, e.g., \citet{osborne2000new}) and model predictive control (see, e.g., \citet{garcia1989model}), {\em (ii)} as a building block for developing globally convergent algorithms by the approach of path-following as is done in interior-point methods (see, e.g., \citet{nesterov1994interior}), and {\em (iii)} the need to address multi-objective optimization, for instance, finding the Pareto frontier of a two-objective optimization problem. 
    Depending on the assumptions made, the goal may be to find global/local optimal solutions or Karush–Kuhn–Tucker (KKT) points of the problem $P(\lambda)$ for $\lambda \in \Lambda$. 
    
    In the rest of the paper, we will focus on a more specific problem, in which we assume that: {\em (i)} the dependence on $\lambda$ is linear, that is, $F(x, \lambda)$ can be written as $f(x) + \lambda \cdot \Omega(x)$, {\em (ii)} both $f(\cdot)$ and $\Omega(\cdot)$ are convex functions with certain properties, and {\em (iii)} the feasible set $S(\lambda)$ is the entire vector space $\bbR^p$ for all $\lambda \in \Lambda$. 
	That is, we focus on the parametric optimization problem: 
	\begin{equation}\label{eq:param-opt-1}
		P(\lambda): \quad F_{\lambda}^* := \min_{x \in \bbR^p} \lrrr{F_{\lambda}(x) := f(x) + \lambda \cdot \Omega(x)}, 
	\end{equation}
	where $f(\cdot): \bbR^p \rightarrow \bbR$ and $\Omega(\cdot): \bbR^p \rightarrow \bbR$ are twice-differentiable functions such that $f(\cdot)$ is $\mu$-strongly convex for some $\mu \ge 0$ and $\Omega(\cdot)$ is $\sigma$-strongly convex for some $\sigma > 0$, both with respect to the $\ell_2$-norm (denoted by $\vvt{\cdot}$ herein). 
	For any $\lambda > 0$, let 
	\begin{equation}\label{eq:reg-path}
		x(\lambda) := \arg \min_{x \in \bbR^p} F_{\lambda}(x)
	\end{equation}
	denote the unique optimal solution of $P(\lambda)$ defined in \cref{eq:param-opt-1}. 
	We are interested in the problem of (approximately) computing the set of optimal solutions $\lrrr{x(\lambda): \lambda \in \Lambda}$ where $\Lambda = [\lambda_{\min}, \lambda_{\max}]$ is the set of interest for some $0 < \lambda_{\min} < \lambda_{\max}$, and we also refer to this set of solutions as the \emph{(exact) solution path}. 
	An important set of problems in practice and a popular line of research involves computing the solution path of regularized machine learning problems, including the LASSO as in \citet{efron2004least,osborne2000new} and the SVM problem as in \citet{hastie2004entire}. 
	In these works, algorithms are designed to compute the exact piecewise-linear solution path. 
	Also, in the context of interior-point methods for constrained convex optimization (see, for instance, \citet{nesterov1994interior} and \citet{renegar2001mathematical}), $f(\cdot)$ represents the objective function and $\Omega(\cdot)$ represents the barrier function induced by the constraints of the original problem. Note that the application to interior-point methods requires a slightly more general variant of problem \eqref{eq:param-opt-1} where $\Omega(\cdot): \bbR^p \rightarrow \bbR \cup \{+\infty\}$.
	 Interior-point methods generally start with the problem $P(\lambda)$ for a moderately large $\lambda_0$ and terminate when $\lambda_k < \delta$ for some small enough positive threshold $\delta$.
  
	Recently, there has been growing interest in developing algorithms for computing an approximate solution path of a generic problem such as \eqref{eq:reg-path}. \citet{rosset2004tracking} consider applying exact Newton steps on prespecified grids, and \citet{ndiaye2019safe} consider adaptive methods to discretize the interval $[\lambda_{\min}, \lambda_{\max}]$, for example.
	These grid search type methods, which discretize the interval $[\lambda_{\min}, \lambda_{\max}]$ and subsequently solve a sequence of individual optimization problems, take a very black-box approach.
	A natural question is: Can we ``open the black-box" by developing a methodology that better exploits the structure of the solution path? 
	We answer this question positively by introducing a differential equation perspective to analyze the solution path, which enables us to better reveal and exploit the underlying structure of the solution path. 
	This deeper understanding enables us to build more efficient algorithms and present improved computational guarantees. 
	
	In particular, we derive an ordinary differential equation with an initial condition whose solution is the exact solution path of \eqref{eq:param-opt-1}. The dynamics of the ODE that we derive resemble, but are distinct from, the dynamics of a path-wise version of Newton's method.
	Based on the ODE, we propose efficient algorithms to generate \emph{approximate solution paths} $\hat{x}(\lambda): \lambda \in [\lambda_{\min}, \lambda_{\max}] \to \bbR^p$ and provide the corresponding complexity analysis. 
	The metric we consider is the $\ell_2$-norm of the gradient of the regularized problem, namely $\vvt{\nabla F_{\lambda}(\hat{x}(\lambda))}_2$, and we use the largest norm along the approximate path $\sup_{\lambda \in \Lambda} \vvt{\nabla F_{\lambda}(\hat{x}(\lambda))}_2$ to represent the accuracy of an approximate path $\hat{x}(\lambda)$ (as formally defined in \cref{def:accuracy}). 
	To analyze the computational cost of our proposed algorithms, we consider the oracle complexity -- either in terms of full Hessian or Hessian-vector product/gradient evaluations -- to obtain an $\epsilon$-accurate solution path. Note that considering the oracle complexity is in contrast to other works that consider the number of individual optimization problems that need to be solved (see, for example, \citet{giesen2012approximating,ndiaye2019safe}), as well as the number of ordinary differential equations that need to be solved (see, for example, \citet{zhou2015path}). 
    
	\subsection{Contributions}\label{sec:contribution}
	The first set of contributions of this paper concern the perspective of the solution path of \eqref{eq:param-opt-1} from an ordinary differential equation point of view. 
	We derive an ordinary differential equation with an initial condition whose solution is the solution path of \eqref{eq:param-opt-1}, based on the first-order optimality conditions of \eqref{eq:param-opt-1}. 
	With this observation, we propose a novel and efficient way to approximate the entire solution path. 
	Our derivation does not rely on the special structure of the optimization problem, like existing results in the solution path of LASSO or SVM problems, and holds for general objective functions. 

	The second set of contributions of this paper concern the design of efficient algorithms and the corresponding oracle complexity analysis. 
	Classical error analysis of numerical ordinary differential equation methods \citep{gautschi2011numerical, scott2011numerical} provides only asymptotic results, and the global error has an exponential dependency on the Lipschitz constant and the length of the time period. 
	In contrast, we design new update schemes to compute an approximate solution path and develop nonasymptotic complexity bounds. 
	In particular, we apply a semi-implicit Euler method on the ordinary differential equation in order to compute the approximate optimal solutions under a finite set of penalty coefficients. 
	Then, we incorporate linear interpolation, which was usually missing either in practice or in the lower bound complexity analysis, to generate nearly optimal solutions under other penalty coefficients within the range of parameter values of interest.
	The two-step algorithm guarantees an $\epsilon$-accurate solution path within at most $\mathcal{O}(\frac{1}{\epsilon})$ gradient and Hessian evaluations as well as linear equation system solves.
	When the objective function has higher-order smoothness properties, we modify the traditional trapezoid method in numerical differential equations and design a new update scheme, which guarantees an $\epsilon$-path within at most $\mathcal{O}(\frac{1}{\sqrt{\epsilon}})$ Hessian evaluations. 
	It is important to emphasize that the complexity results in this paper are in terms of the number of operations (for example, Hessian evaluations), rather than the number of sub-problems that need to be solved (for example, solving a single numerical ODE or individual optimization problems) as has been studied in prior work such as \citet{giesen2012approximating, zhou2015path} and \citet{ndiaye2019safe}. 
	\newtext{Further extensions of our proposed algorithms to Runge-Kutta methods are possible but we omit them for brevity.}
	We also provide a detailed computational evaluation of our algorithms and existing methods, including several experiments with synthetic data, the breast cancer dataset \cite{Dua:2019}, and the leukemia dataset \cite{golub1999molecular}.
	
	The third set of contributions of the paper concerns second-order conjugate gradient type methods and computational guarantees in the presence of inexact gradient and Hessian oracles, as well as approximate linear equation solvers. 
	When the dimension of the problem is high, computing the Hessian and/or solving linear systems becomes a computational bottleneck.
	To avoid this, one would like an algorithm that only requires approximate directions at each iteration. 
	We first consider the case where the (absolute) numerical error incurred in the calculation of a new direction $d_k$ is bounded by some $\delta_k > 0$. 
	We show that our algorithms are robust to numerical error in the sense that the additional errors of inexact directions do not accumulate and do not depend on the condition number. 
	We extend the complexity analysis to the case where the numerical error $\delta_k$ has a uniform upper bound $\alpha \epsilon$ for $\alpha \in (0, 1)$ and show that the Euler method maintains $\mathcal{O}(\frac{1}{\epsilon})$ complexity, and the trapezoid method maintains $\mathcal{O}(\frac{1}{\sqrt{\epsilon}})$ complexity when there is higher-order smoothness. 
	We then propose variations of the algorithms mentioned before that only require gradient and Hessian-vector product oracles, rather than gradient and Hessian oracle, as well as a linear system solver. 
	We also leverage the previous analysis in the case of inexact directions in order to provide computational complexity results for the second-order conjugate gradient-type algorithms, which have the same order of $\epsilon$ as the results for the exact methods mentioned above.
	Our results demonstrate that our algorithms are more robust and second-order conjugate gradient variations require less computational cost compared to existing methods. 
	
	\subsection{Related Literature}\label{sec:literature}
    We now discuss previous work related to our algorithm and analysis from three distinct aspects. 
    	
    \paragraph{Other path methods and comparison of results}
    As previously mentioned, for the LASSO and SVM problems, the exact solution path is piecewise linear and can be computed by the path following methods such as the least angle regression (LARS) algorithm proposed by \citet{efron2004least,hastie2004entire} and \citet{osborne2000new}. 
    Additional problems whose solution paths are piecewise linear are considered by \citet{rosset2007piecewise}, and \citet{mairal2012complexity} showed that the number of breakpoints in the solution path can be exponential in the number of data points. 
    Generalized linear regression problems with $\ell_1$ regularization are considered by \citet{yuan2009efficient} and \citet{zhou2014generic} via LARS based methods. 
    Another line of work focuses on computing approximate solution paths for specific problems, including the elastic net in \citet{friedman2010regularization}, the SVM problem in \citet{bach2006considering} and \cite{giesen2012approximating}, matrix completion with nuclear norm regularization in \citet{mazumder2010spectral}, and other structural convex problems in \citet{giesen2012regularization} and  \citet{loosli2007regularization}. 
    For problems with nonconvex but coordinate decomposable regularization functions, a coordinate descent-based algorithm was considered by \citet{mazumder2011sparsenet} and \citet{wang2014optimal}. 
    Closest to our problem set up, \citet{rosset2004tracking} considered a general problem when $f(\cdot)$ and $\Omega(\cdot)$ have third-order derivatives and provided an algorithm which applied exact Newton steps on equally spaced grids starting from the optimal solution of the nonregularized problem. 
    The lower bound complexity analysis when the approximate solution path is limited to a piecewise constant function is considered by \citet{giesen2012approximating, ndiaye2019safe}. 
    	
    \paragraph{Related global complexity analysis of second-order methods}
    Newton-like methods are an important class of algorithms in optimization. 
    Some notable lines of work include interior-point methods (\citet{nesterov1994interior, renegar2001mathematical}) and applications in regression problems (\citet{kim2007interior}) as well as the Levenberg-Marquardt method (\citet{more1978levenberg}). 
    In practice, techniques are often incorporated to ensure global convergence, such as line searches (\citet{ralph1994global}) and trust-region techniques (\citet{conn2000trust}). 
    The global complexity analysis of Newton and higher-order methods with regularization has also received much recent interest, such as the work of \citet{nesterov2018implementable, nesterov2006cubic} and \citet{polyak2009regularized} and the references therein. In our paper, we make similar types of assumptions as in the global complexity analysis of regularized second and higher-order methods and we also prove global complexity results for the class of second-order algorithms considered herein. 
    
    \paragraph{Related work on differential equations and optimization methods}
    Early works, including the work of \citet{alvarez2001inertial} and \citet{alvarez2002second}, analyzed inertial dynamical systems driven by gradient and Newton flows with application to optimization problems. 
    Newton-like dynamic systems of monotone inclusions with connections to Levenberg-Marquardt and regularized Newton methods were also analyzed by \citet{abbas2014newton, attouch2011continuous} and \citet{bot2016second}. 
    Due to the recent popularity of the accelerated gradient descent method in the machine learning and optimization communities, the limiting dynamics of different optimization methods and their discretization have thus received much renewed interest in recent years; see \citet{attouch2018fast, scieur2017integration, su2014differential, wibisono2016variational, zhang2018direct} and the references therein. 
	
	\subsection{Organization}
	The paper is organized as follows. 
    In \Cref{sec:ode}, we derive the ordinary differential equation with an initial condition whose solution is the exact solution path \eqref{eq:reg-path} and provide the existence and uniqueness of the solution of the differential equation. 
    In \Cref{sec:discretizations}, we leverage the ODE to develop a numerical algorithm to compute the approximate solution path of \eqref{eq:param-opt-1}, and derive the corresponding complexity analysis in \cref{thm:err-ctrl-exp}. 
    In \Cref{sec:multi}, we propose a multistage method that is beneficial when the functions $f(\cdot)$ and $\Omega(\cdot)$ have higher order smoothness, and we also provide its complexity analysis in \cref{thm:err-ctrl-trapezoid}. 
    In \Cref{sec:inexact}, we extend our results in the presence of inexact oracles, and as a direct application we propose second-order conjugate gradient variants of the aforementioned algorithms, which avoid exact Hessian evaluations and exact solutions of linear systems. 
    \Cref{sec:experiments} contains a detailed computational experiments of the proposed algorithms and grid search methods on real and synthetic datasets. 

    \subsection{Notation}
    For a positive integer $n$, let $[n] := \{1, \dots, n\}$. 
    For a vector-valued function $y(t): \bbR \rightarrow \bbR^p$ that can be written as $y(t) = (y_1(t), \dots, y_p(t))$, we say that $y(\cdot)$ is differentiable if $y_i(\cdot)$ is differentiable for all $i = 1, \dots, p$ and let $\frac{\dd y}{\dd t}$ be the derivative of $y(t)$, namely $\frac{\dd y}{\dd t} = (\frac{\dd y_1}{\dd t}, \dots, \frac{\dd y_p}{\dd t})$. 
    Let ${\bm 1}_p$ and ${\bm 1}_{p \times p}$ denote the $p$-dimensional all-ones vector and the $p \times p$ all-ones matrix, respectively. 
    Throughout the paper, we fix the norm $\vvta{\cdot}$ on $\bbR^p$ to be the $\ell_2$-norm, which is defined by $\vvta{x} := \vvta{x}_2 = \sqrt{x^T x}$.
    Also, in a slight abuse of notation, we use $\vvta{\cdot}$ to represent the operator norm, i.e., the induced $\ell_2$-norm $\vvta{\cdot}_2$ on $\bbR^{n \times p}$, which is defined by $\vvta{A} := \vvta{A}_{2,2} = \max_{\vvta{x}_2 \le 1} \vvta{A x}_2$.

\section{Ordinary Differential Equation Characterization of Solution Path}\label{sec:ode}

    Let us begin by describing a differential equations perspective on the solution path \cref{eq:reg-path} of \cref{eq:param-opt-1} that will prove fruitful in the development of efficient computational methods.
	First, we introduce a reparameterization in terms of an auxiliary variable $t \geq 0$ (thought of as ``time''), whereby for a given $T > 0$ we introduce functions $\lambda(\cdot) : [0, T] \to [\lambda_{\min}, \lambda_{\max}]$ and $\xi(\cdot) : [\lambda_{\min}, \lambda_{\max}] \to \bbR$ such that: $\xi(\cdot)$ is Lipschitz, $\lambda(\cdot)$ is differentiable on $(0, T)$, and we have $\frac{\dd \lambda}{\dd t} = \xi(\lambda(t))$ for all $t \in (0, T)$. In a slight abuse of notation, we define the path with respect to $t$ as $\lrrr{x(t) := x(\lambda(t)) : t \in [0, T]}$.
	Now notice that, for any $t \in [0, T]$, the first-order optimality condition for the problem $P(\lambda(t))$ states that $\nabla f(x(t)) + \lambda(t) \nabla \Omega(x(t)) = 0$. 
	By differentiating both sides of the previous equation with respect to $t$, we have $\nabla^2 f(x(t)) \cdot \frac{\dd x}{\dd t} + \nabla \Omega(x(t)) \cdot \frac{\dd \lambda}{\dd t} + \lambda(t) \nabla^2 \Omega(x(t)) \cdot \frac{\dd x}{\dd t} = 0$. 
	Rearranging the above and again using $\frac{\dd \lambda}{\dd t} = \xi(\lambda(t))$ yields $\frac{\dd x}{\dd t} = - \lr{\nabla^2 f(x(t)) + \lambda(t) \nabla^2 \Omega(x(t))}^{-1} \xi(\lambda(t)) \nabla \Omega(x(t))$. 
	Then, applying the fact that $\nabla f(x(t)) + \lambda(t) \nabla \Omega(x(t)) = 0$ yields $\frac{\dd x}{\dd t} = \lr{\nabla^2 f(x(t)) + \lambda(t) \nabla^2 \Omega(x(t))}^{-1} \cdot \frac{\xi(\lambda(t))}{\lambda(t)} \nabla f(x(t))$. 
	Thus, we arrive at the following autonomous system 
	\begin{equation}\label{eq:reg-dynamic-2}
	    \frac{\dd \lambda}{\dd t} = \xi(\lambda), \quad
		\frac{\dd x}{\dd t} = v(x, \lambda) := \lr{\nabla^2 f(x) + \lambda \nabla^2 \Omega(x)}^{-1} \frac{\xi(\lambda)}{\lambda} \nabla f(x), 
	\end{equation}
	for $t \in [0, T]$. 
	
	By considering specific choices of $\xi(\cdot)$ and $\Omega(\cdot)$, the system \eqref{eq:reg-dynamic-2} generalizes some previously studied methodologies in parameteric optimization.
	First, consider the scenario with an equally spaced discretization of the interval $[0, T]$, namely $t_k = k \cdot h$ for some fixed step-size $h > 0$. 
	Thus, the sequence $\lambda_k := \lambda(t_k)$ is approximately given by $\lambda_{k+1} \approx \lambda_k + h \cdot \xi(\lambda_k)$. 
	Intuitively, the choice of $\xi(\cdot)$ controls the dynamic of $\lambda(\cdot)$ and generalizes some previously considered sequences $\{\lambda_k\}$ for the problem \eqref{eq:param-opt-1}. 
	For example, letting $\xi(\lambda) \equiv 1$ we recover the arithmetic sequence in \citet{rosset2004tracking} and letting $\xi(\lambda) \equiv -\lambda$ we recover the geometric sequence in \citet{ndiaye2019safe}. 
	In addition, consider the special case where $\Omega(x) = \frac12 \vvt{x}^2$. Then the dynamic for $x(t)$ in \cref{eq:reg-dynamic-2} is similar to the limiting dynamic for the proximal Newton method (also known as the Levenberg-Marquardt regularization procedure \cite{marquardt1963algorithm} for convex optimization problems). 
	The property of a similar dynamic of monotone inclusion is analyzed in \citet{attouch2016dynamic, attouch2011continuous}, which includes finding zero of the gradient of a convex function. 

    Before developing and presenting algorithms designed to compute the approximate solution path based on \eqref{eq:reg-dynamic-2}, we first verify that the system \eqref{eq:reg-dynamic-2} has a unique trajectory. 
    The following proposition states conditions on $f(\cdot)$ and $\Omega(\cdot)$ such that $v(\cdot, \cdot)$ defined in \cref{eq:reg-dynamic-2} is continuous in $\lambda \in [\lambda_{\min}, \lambda_{\max}]$ and is uniformly $L_v$-Lipschitz continuous with respect to $x$, namely, $\vvta{v(x_1, \lambda) - v(x_2, \lambda)} \le L_v \vvta{x_1 - x_2}$ for any $\lambda \in [\lambda_{\min}, \lambda_{\max}], x_1, x_2 \in \bbR^p$. 
    This uniform Lipschitz property ensures that the above system has a unique trajectory, which therefore coincides with the solution path defined in \eqref{eq:reg-path}.

	\begin{proposition}[Theorem 5.3.1  of \citet{gautschi2011numerical}]\label{prop:ode}
	    Suppose that $\nabla^2 f(\cdot)$, $\nabla^2 \Omega(\cdot)$, $\nabla f(\cdot)$, $f(\cdot)$ are all $L$-Lipschitz continuous, $f(\cdot)$ is $\mu$-strongly convex for $\mu \geq 0$, and $\Omega(\cdot)$ is $\sigma$-strongly convex for $\sigma > 0$. Suppose further that $\xi(\cdot)$ is Lipschitz continuous on $[\lambda_{\min}, \lambda_{\max}]$ and satisfies $\vta{\frac{\xi(\lambda)}{\lambda}} \le C$ for all $\lambda \in [\lambda_{\min}, \lambda_{\max}]$. Then, it holds that $v(\cdot, \cdot)$ defined in \cref{eq:reg-dynamic-2} is continuous in $\lambda \in [\lambda_{\min}, \lambda_{\max}]$ and is uniformly $L_v$-Lipschitz continuous with $L_v = \frac{LC}{\mu+\lambda_{\min}\sigma} + \frac{L^2 C (1+\lambda_{\max})}{(\mu+\lambda_{\min}\sigma)^2}$, 
	    and \cref{eq:reg-dynamic-2} has a unique trajectory $(\lambda(t), x(t))$ for $t \in [0, T]$.
	\end{proposition}

\section{Discretization and Complexity Analysis}\label{sec:discretizations}
	
    In this section, we present algorithms to compute an approximate solution path based on discretizations of \cref{eq:reg-dynamic-2}, along with the corresponding complexity analysis. The primary error metric we consider is the $2$-norm of the gradient across the entire interval $[\lambda_{\min}, \lambda_{\max}]$ as formally presented in \cref{def:accuracy}. 
    \begin{definition}\label{def:accuracy}
		An approximate solution path $\hat{x}(\cdot) : [\lambda_{\min}, \lambda_{\max}] \to \bbR^p$ to the parametric optimization problem \cref{eq:param-opt-1} has accuracy $\epsilon \geq 0$ if $\vvt{\nabla F_{\lambda}(\hat{x}(\lambda))} \le \epsilon$ for all $\lambda \in [\lambda_{\min}, \lambda_{\max}]$. 
	\end{definition}
    Notice that the strong convexity of the objective function $F_\lambda(\cdot)$ for all $\lambda > 0$ immediately implies that an $\epsilon$-accurate solution path $\hat{x}(\cdot)$ also has the optimality gap guarantee, which is $F_{\lambda}(\hat{x}(\lambda)) - F_\lambda^\ast \leq \frac{\epsilon^2}{2(\mu + \lambda \sigma)} \leq \frac{\epsilon^2}{2(\mu + \lambda_{\min} \sigma)}$, 
    for all $\lambda \in [\lambda_{\min}, \lambda_{\max}]$. 
    \cref{alg:meta} below presents a two-step ``meta-algorithm'' for computing an approximate solution path $\hat{x}(\cdot)$. 
    Inspired by numerical methods to solve ordinary differential equations, we first design several update schemes to iteratively update $(x_k, \lambda_k)$ by exploiting the dynamics in \cref{eq:reg-dynamic-2}. We use the function $\psi(\cdot, \cdot) : \bbR^p \times [\lambda_{\min}, \lambda_{\max}] \to \bbR^p \times [\lambda_{\min}, \lambda_{\max}]$ to denote a generic update rule in the meta-algorithm below, and we consider several different specific examples herein. Then we apply an interpolation method $\mathcal{I}(\cdot)$ to resolve the previously computed sequence of points into an approximate path $\hat{x}(\cdot)$ over $[\lambda_{\min}, \lambda_{\max}]$.
    \begin{algorithm}
        \caption{Meta-algorithm for computing an approximate solution path $\hat{x}(\cdot)$}
		\label{alg:meta}
		\begin{algorithmic}[1]
		    \STATE{\textbf{Input:} initial point $x_0 \in \bbR^p$, total number of iterations $K \geq 1$, update rule $\psi(\cdot, \cdot)$, and interpolation method $\mathcal{I}(\cdot)$}
		    \STATE{Initialize regularization parameter $\lambda_0 \gets \lambda_{\max}$}
		    \FOR{$k = 0, \dots, K-1$}
		        \STATE{Update $(x_{k+1}, \lambda_{k+1}) \gets \psi(x_k, \lambda_k)$}
		    \ENDFOR
		    \RETURN{$\hat{x}(\cdot) \gets \mathcal{I}\lr{{\{(x_k, \lambda_k)\}}_{k=1}^K}$}
		\end{algorithmic}
    \end{algorithm}
    
    We develop oracle complexity results for different update schemes and interpolation methods in terms of the number of gradient computations, Hessian computations, and linear system solves required to compute an $\epsilon$-accurate approximate path. 
	In this section, we will stick to a simple version of \cref{alg:meta} based on the application of the semi-implicit Euler method and linear interpolation to specify the update rule $\psi(\cdot, \cdot)$ and the interpolation method $\mathcal{I}(\cdot)$. 
    In particular, the semi-implicit Euler discretization of \cref{eq:reg-dynamic-2} is
    \begin{equation}\label{eq:gen-euler1}
    	\lambda_{k+1} = \lambda_k + h \cdot \xi(\lambda_k), \quad x_{k+1} = x_k + h \cdot v(x_k, \lambda_{k+1}), 
	\end{equation}
	and the linear interpolation $\mathcal{I}_{\textnormal{linear}}(\cdot): {\{(x_k, \lambda_k)\}}_{k=0}^K \to \hat{x}(\cdot)$ is defined by $\hat{x}(\lambda) := \alpha x_k + (1-\alpha) x_{k+1}$ with $\alpha = \frac{\lambda-\lambda_{k+1}}{\lambda_k-\lambda_{k+1}}$ for all $\lambda \in [\lambda_{k+1}, \lambda_k]$ and $k \in \{0, \ldots, K - 1\}$. 
	
	\paragraph{The algorithm with the exponential decaying parameter sequence}
	Recall the update rule \cref{eq:gen-euler1}, we can see that the function $\xi(\cdot)$ (or equivalently, $\lambda(\cdot)$) is still to be determined. 
    In practical cases, the value of $\lambda$ usually decreases exponentially from $\lambda_{\max}$ to $\lambda_{\min}$. 
    This choice of penalty scale parameters $\lrrr{\lambda_k}$ arises in the solution path for linear models, see \citet{friedman2010regularization}, and the interior-point method, see \citet{renegar2001mathematical}. 
	Although our analysis holds for a broad class of $\lambda(\cdot)$, we first present the version with an exponentially decaying parameter sequence, namely $\lambda(t) = e^{-t} \lambda(0)$. 
	This specific version of \cref{alg:meta} is formally described in \cref{alg:euler}. 
	
	\begin{algorithm}
	    \caption{Euler method for computing an approximate solution path $\hat{x}(\cdot)$}
		\label{alg:euler}
		\begin{algorithmic}[1]
		    \STATE{\textbf{Input:} initial point $x_0 \in \bbR^p$, total number of iterations $K \geq 1$}
		    \STATE{Initialize regularization parameter $\lambda_0 \gets \lambda_{\max}$, set step-size $h \gets 1 - (\frac{\lambda_{\min}}{\lambda_{\max}})^{\frac{1}{K}}$}
		    \FOR{$k = 0, \dots, K-1$}{
		        \STATE{$\lambda_{k+1} \gets \lambda_k - h \cdot \lambda_k$}
		        \STATE{$x_{k+1} \gets x_k - h \lr{\nabla^2 f(x_k) + \lambda_{k+1} \nabla^2 \Omega(x_k)}^{-1} \nabla f(x_k)$}
		    }
		    \ENDFOR
		    \RETURN{$\hat{x}(\cdot) \gets \mathcal{I}_{\textnormal{linear}}\lr{{\{(x_k, \lambda_k)\}}_{k=1}^K}$ according to linear interpolation}
		\end{algorithmic}
	\end{algorithm}
    
	Before going further into detailed analysis, we first state the computational guarantee of \cref{alg:euler}. 
    In our complexity analysis, we make the following smoothness assumptions for $f(\cdot)$ and $\Omega(\cdot)$.
    
    \begin{assumption}\label{as:lips-hess}
		In addition to $\mu$-strong convexity of $f(\cdot)$ and $\sigma$-strong convexity of $\Omega(\cdot)$, these functions have $L$-Lipschitz gradients and Hessians, where $L > 0$ is an upper bound on the four relevant Lipschitz constants. 
		Furthermore, we assume that $f^\ast := \min_x f(x) > -\infty$ and that $G > 0$ is an upper bound on the norm of the gradients of $f(\cdot)$ and $\Omega(\cdot)$ on the level set $\{x \in \bbR^p : f(x) \leq f(x_0)\}$.
	\end{assumption}
	
	\Cref{thm:err-ctrl-exp} is our main result concerning the complexity of \cref{alg:euler} and demonstrates that in terms of the accuracy parameter $\epsilon$, \cref{alg:euler} requires $\mathcal{O}(1/\epsilon)$ iterations to compute an $\epsilon$-accurate solution path.
	
	\begin{theorem}\label{thm:err-ctrl-exp}
		Suppose that \cref{as:lips-hess} holds, let $\epsilon > 0$ be the desired accuracy, and suppose that the initial point $x_0$ satisfies $\vvt{\nabla F_{\lambda_{\max}}(x_0)} \le \frac{\epsilon}{4}$. Let $T := \log(\lambda_{\max}/\lambda_{\min})$, let $\tau = \max \{ \frac{1 + \lambda_{\min}}{\mu + \lambda_{\min} \sigma}, \frac{1 + \lambda_{\max}}{\mu + \lambda_{\max} \sigma} \} $, and let
		\begin{equation}\label{eq:err-ctrl-exp}
			K_{\tn{E}} := \left\lceil \max \lrrr{2T, \frac{\sqrt{LG} \tau T}{\sqrt{3} }, \frac{4 (f(x_0) - f^\ast) \tau L T}{\epsilon}, \frac{2 \sqrt{L} (\tau G + 1) T}{\sqrt{\epsilon}}} \right\rceil. 
		\end{equation}
		If the total number of iterations $K$ satisfies $K \geq K_{\tn{E}}$, then \cref{alg:euler} returns an $\epsilon$-accurate solution path.
	\end{theorem}
	
	\begin{remark}\label{rmk:grid-search}
	    Grid search type methods for computing approximate solution paths are proposed in \citet{giesen2012approximating, ndiaye2019safe}, and we will follow the analysis of the latter, which considers more general cases. 
	    To ensure that the function value gap satisfies $h(x) - h^* \le \epsilon'$, we require the number of grid points $K = \frac{\sqrt{\tau} G T}{\sqrt{\epsilon'}}$. 
	    For $L$-smooth function $h(\cdot)$, $h(x) - h^* \le \frac{\epsilon^2}{2L}$ implies $\vvta{h(x)} \le \epsilon$, which is the goal in our paper. 
	    Therefore, we need to set $\epsilon' = \frac{\epsilon^2}{2L}$, and hence we have the number of grid points is $K = \frac{\sqrt{\tau L} G T}{\epsilon}$ when the desired accuracy of the inner problem is set to $\epsilon_c = \frac{\epsilon'}{2}$. 
	    Using the exact Newton method and the last grid point as a warm start to solve the inner problem, \citet{karimireddy2018global} implies that the inner complexity is $(\tau L)^2 \log 2$, and therefore the total complexity is $\frac{(\tau L)^{5/2} G T \log 2}{\epsilon}$. 
	\end{remark}
	
	From the results in \cref{thm:err-ctrl-exp,rmk:grid-search}, we can see that both our results and the grid search method have a complexity of order $\mathcal{O}(\frac{1}{\epsilon})$. 
	In most practical cases, $f(x_0) - f^*$ is smaller than $(\tau L)^{3/2} G \log 2$, which implies that the constant in \cref{eq:err-ctrl-exp} is better. 
	In \Cref{sec:multi}, we will propose algorithms utilizing the smoothness of $f(\cdot)$ and $\Omega(\cdot)$ and achieve better complexity results. 
	
	\subsection{Semi-Implicit Euler Update Scheme}
	Herein we provide the computational guarantee of the semi-implicit Euler update scheme defined in \cref{eq:gen-euler1}, where the main object we concern is the accuracy at each grid point. 
	Let $r_k := \vvt{\nabla F_{\lambda_k}(x_k)}$ denote the accuracy at $(x_k, \lambda_k)$. 
	We consider a broad family of $\xi(\cdot)$ in the following analysis, although in \cref{alg:euler} and the corresponding complexity analysis in \cref{thm:err-ctrl-exp}, we only consider the special scenario that $\lambda(t) = e^{-t} \lambda(0)$. 
	We first present the computation guarantees of Taylor expansion approximations given the Lipschitz continuity of Hessian in \cref{lm:tay-exp}.  

	\begin{lemma}[Lemma 1 in \citet{nesterov2006cubic}]\label{lm:tay-exp}
	    Suppose $\phi(\cdot): S \to \bbR$ has $L$-Lipschitz Hessian for some convex set $S \subseteq \bbR^d$, then the following inequalities hold for all $x, y \in S$: 
		\begin{enumerate}[(i)]
	        \item \label{it:tay-exp-1} $\vvt{\nabla \phi(y) - \nabla \phi(x) - \nabla^2 \phi(x) (y-x)} \le \frac{L}{2} \vvt{y-x}^2$. 
	        \item \label{it:tay-exp-2} $\vvt{\phi(y) - \phi(x) - \nabla \phi(x)^T (y-x) - \frac12 (y-x)^T \nabla^2 \phi(x) (y-x)} \le \frac{L}{6} \vvt{y-x}^3$. 
		\end{enumerate}
	\end{lemma}

	Based on the results in \cref{lm:tay-exp}, we provide the local analysis of $r_k$, that is, how the norm of the gradient at each grid point accumulates. 
	\cref{lm:gen-err-2} provides an upper bound on $\vvt{r_{k+1}}$ based on $\vvt{r_k}$, which represents the accuracy in the previous iteration. 

	\begin{lemma}\label{lm:gen-err-2}
		Suppose \cref{as:lips-hess} holds, discretization \cref{eq:gen-euler1} has the following guarantee for all $k \ge 0$: 
		\begin{equation*}
		    r_{k+1} \le \frac{\lambda_{k+1}}{\lambda_k} \cdot r_k + h^2 \cdot \frac{L(1+\lambda_{k+1})}{2} \vvt{v(x_k, \lambda_{k+1})}^2. 
		\end{equation*}
	\end{lemma}
	\begin{proof}
	    It holds that 
		\begin{align*}
			r_{k+1} &= \vvt{\nabla f(x_{k+1}) + \lambda_{k+1} \nabla \Omega(x_{k+1})} \\
		   	& \le \vvt{\nabla f(x_{k+1}) - \nabla f(x_k) - \nabla^2 f(x_k) (x_{k+1} - x_k)} \\ 
			& \phantom{\le} + \vvt{\lambda_{k+1} \lr{\nabla \Omega(x_{k+1}) - \nabla \Omega(x_k) - \nabla^2 \Omega(x_k) (x_{k+1} - x_k)}} \\
			& \phantom{\le} + \vvt{\lambda_{k+1} \lr{\nabla \Omega(x_k) + \nabla^2 \Omega(x_k) (x_{k+1}-x_k)} + \nabla f(x_k) + \nabla^2 f(x_k) (x_{k+1}-x_k)} \\
			& \le \frac{L}{2} \vvt{x_{k+1} - x_k}^2 + \frac{\lambda_{k+1} L}{2} \vvt{x_{k+1} - x_k}^2 + \frac{\lambda_{k+1}}{\lambda_k} \vvt{\nabla f(x_k) + \lambda_k \nabla \Omega(x_k)}\\
			& = \frac{\lambda_{k+1}}{\lambda_k} \cdot r_k + h^2 \cdot \frac{L(1+\lambda_{k+1})}{2} \vvt{v(x_k, \lambda_{k+1})}^2, 
		\end{align*}
		where the first inequality is true due to the triangle inequality, and in the second inequality, for the first two terms, we apply \cref{it:tay-exp-1} in \cref{lm:tay-exp}, and for the third terms, they are equal to each other. 
	\end{proof}

    \Cref{lm:gen-err-2} provides the first technical result of the semi-implicit Euler update scheme. 
    When \cref{as:lips-hess} holds, \cref{eq:gen-euler1} has the computation guarantee $\tfrac{r_{j+1}}{\lambda_{j+1}} \le \tfrac{r_j}{\lambda_j} + \mathcal{O}(h^2)$. 
    After telescoping the inequalities for $j = 0, 1, \dots, k-1$ for $k \le K = T/h$ we have $r_k \sim \mathcal{O}(h)$. 
	Therefore, one would expect a uniform $\mathcal{O}(h)$ bound on all accuracy $r_k$ at points $x_k$ for all $k = 0, \dots, K$. 
	We formalize the idea in the following lemma. 

	\begin{lemma}\label{lm:global-rk}
	    Suppose \cref{as:lips-hess} holds, $\lambda_{k+1} \ge \lambda_{\min}$, $\xi(\lambda_j) < 0$ for all $j \le k$, and step-size $h$ satisfies the following condition for all $0 \le j < k$: 
		\begin{equation}\label{eq:global-h}
			h \le \min \lrrr{\frac{\lambda_j}{-2\xi(\lambda_j)}, \sqrt{\frac{3\lambda_j(\mu+\lambda_{j+1}\sigma)^2}{-\xi(\lambda_j) L G}}}.  
		\end{equation}
		Then the sequence $\{(x_k, \lambda_k)\}_{k=0}^K$ generated by update scheme \cref{eq:gen-euler1} satisfies $f(x_{k+1}) \le f(x_k)$, and the corresponding accuracy $\{r_k\}_{j=0}^K$ has the following guarantee: 
		\begin{equation}\label{eq:global-param}
			\frac{r_k}{\lambda_k} \le \frac{r_0}{\lambda_0} + 2 h L (f(x_0) - f(x_k)) \cdot \max_{j \in [k-1]} \lrrr{\frac{- (1 + \lambda_{j+1}) \xi(\lambda_j)}{\lambda_j \lambda_{j+1} (\mu + \lambda_{j+1} \sigma)}}. 
		\end{equation}
	\end{lemma}
	\begin{proof}
	    We will use induction to show that $f(x_{j+1}) \le f(x_j)$, which implies that $x_k \in S_{x_0}$ for all $k \in \{0, \dots, K\}$. 
	    First, suppose $x_j \in S_{x_0}$ for some $j \in \{0, \dots, K-1\}$. 
		Let $d_j$ denote $v(x_j, \lambda_{j+1})$ and we have $\vvt{d_j} \le \frac{G}{\mu+\lambda_{j+1}\sigma}$. 
		Since $h + \frac{\lambda_k}{2\xi(\lambda_k)} \le 0$ and $h^2 \le \frac{3\lambda_j(\mu+\lambda_{j+1}\sigma)^2}{-\xi(\lambda_j) L G} \le \frac{3 \lambda_j (\mu + \lambda_{j+1} \sigma)}{-\xi(\lambda_k) L \vvt{d_k}}$, it holds that 
		\begin{align*}
		    f(x_{j+1}) & \le f(x_j) + h \cdot \nabla f(x_j)^T d_j + h^2 \cdot \frac12 d_j^T \nabla^2 f(x_j) d_j + h^3 \cdot \frac16 L \vvt{d_j}^3 \\
		    & \le f(x_j) + \frac{h}{4} \cdot \frac{\lambda_j}{\xi(\lambda_j)} \cdot (\mu+\lambda_{j+1}\sigma) \vvt{d_j}^2. 
		\end{align*}
		Since $\xi(\lambda_j) < 0$, it holds that $f(x_{j+1}) \le f(x_j)$ and therefore implies that $x_{j+1} \in S_{x_0}$. 
		Then by induction we conclude that $x_j \in S_{x_0}$ for all $j \le k$. 
		Also, combining the result in \cref{lm:gen-err-2}, for all $j \le k$, it holds that 
		\begin{equation*}
		    \frac{r_{j+1}}{\lambda_{j+1}} \le \frac{r_j}{\lambda_j} + h^2 \cdot \frac{L(1+\lambda_{j+1})}{2\lambda_{j+1}} \vvt{d_j}^2 
			\le \frac{r_j}{\lambda_j} + h \cdot \frac{L(1+\lambda_{j+1})}{2\lambda_{j+1}} \cdot \frac{4(f(x_j)-f(x_{j+1}))}{\frac{\lambda_j}{-\xi(\lambda_j)} \cdot (\mu+\lambda_{j+1}\sigma)}. 
		\end{equation*}
		By taking the summation over $j$ from $0$ to $k$, we obtain \cref{eq:global-param}. 
	\end{proof}

    \Cref{lm:global-rk} provides the computational guarantees of \cref{alg:euler} with any decreasing sequence of $\lrrr{\lambda_k}$. 
    Note that the upper bound of the accuracy at point $x_k$ still involves a constant related to the sequence $\lrrr{\lambda_k}$. 
    We then consider a family of the sequence $\lrrr{\lambda_k}$ and derive a simpler upper bound. 
	\begin{proposition}\label{prop:global-rk}
		Suppose \cref{as:lips-hess} holds, $\lambda_{k+1} \ge \lambda_{\min}$, $-\lambda \le \xi(\lambda) < 0$ for all $\lambda \in [\lambda_{\min}, \lambda_{\max}]$, and step-size $h$ satisfies 
		\begin{equation}\label{eq:global-h-2}
			h \le \min \lrrr{\frac12, \sqrt{\frac{3 }{ \tau^2 L G}}}. 
		\end{equation}
		Then under \cref{as:lips-hess}, discretization \cref{eq:gen-euler1}, it holds that
		\begin{equation}\label{eq:global-param-2}
			r_{k+1} \le r_0 + 2h \tau L (f(x_0) - f(x_{k+1})). 
		\end{equation}
	\end{proposition}
	\begin{proof}
	    Since $\xi(\lambda_j) \in [-\lambda_j, 0)$, we have $\frac{\lambda_j}{-\xi(\lambda_j)} \ge 1$. 
	    Therefore, we have 
        \begin{equation*}
	        \frac{\lambda_j}{-2\xi(\lambda_j)} \ge \frac12 \quad \text{and} \quad  \sqrt{\frac{3\lambda_j(\mu+\lambda_{j+1}\sigma)^2}{-\xi(\lambda_j) L G}} \ge \sqrt{\frac{3(\mu+\lambda_{\min}\sigma)^2}{L G}}, 
		\end{equation*} 
	    and hence condition \cref{eq:global-h-2} implies condition \cref{eq:global-h}. 
	    Also, since $\frac{1 + \lambda}{\mu + \lambda \sigma}$ is monotone in $[\lambda_{\min}, \lambda_{\max}]$, it holds that $\tau = \max_{\lambda \in [\lambda_{\min}, \lambda_{\max}]} \frac{1 + \lambda}{\mu + \lambda \sigma}$. 
	    Therefore, we have 
	    \begin{equation*}
	        \max_{j \in [k-1]} \lrrr{\frac{- (1 + \lambda_{j+1}) \xi(\lambda_j)}{\lambda_j \lambda_{j+1} (\mu + \lambda_{j+1} \sigma)}} \le \frac{1 + \lambda_k}{\lambda_k (\mu + \lambda_k \sigma)} \le \frac{\tau}{\lambda_k}. 
	    \end{equation*}
	    Apply the above inequality to \cref{eq:global-param} we obtain \cref{eq:global-param-2}. 
	\end{proof}
		
	\Cref{prop:global-rk} provides a uniform upper bound on accuracy of all almost-optimal solutions $x_k$. 
	When $\frac{\xi(\lambda)}{\lambda} \in [-1, 0)$ for all $\lambda \in [\lambda_{\min}, \lambda_{\max}]$, \cref{alg:euler} generates a solution sequence $\lrrr{x_k}$ such that $\vvt{\nabla F_{\lambda_k}(x_k)} \sim \mathcal{O}(h) + r_0$. 

	\subsection{Linear Interpolation}
	
	In the last section, we provided a general accuracy analysis at all nearly-optimal solutions $x_k$. 
	Then the next step is to construct the entire path $\hat{x}(\lambda): [\lambda_{\min}, \lambda_{\max}] \rightarrow \bbR^P$ based on these nearly-optimal solutions. 
	In this section, we analyze the second procedure in \cref{alg:euler}, \ie linear interpolation. 
	First, we recall the definition of linear interpolation for $x(\cdot), \lambda(\cdot)$: 
	\begin{equation}\label{eq:linear-int}
		\hat{\lambda}(t) := \alpha \lambda_k + (1-\alpha) \lambda_{k+1}, \quad \hat{x}(t) := \alpha x_k + (1-\alpha) x_{k+1}, 
	\end{equation} 
	where $\alpha := \frac{t_{k+1}-t}{h}$, $t \in [t_k, t_{k+1}]$, and $k \in \{ 0, \dots, K-1 \}$.
	That is, given an arbitrary $\lambda \in [\lambda_{\min}, \lambda_{\max}]$, we first select $t$ such that $\hat{\lambda}(t) = \lambda$, then we output $\hat{x} := \hat{x}(t)$ as a nearly-optimal solution to problem $P(\lambda)$. 
	The following lemma provides a uniform upper bound of $\Vert \nabla F_{\lambda} (\hat{x}(\lambda)) \Vert$ for all $\lambda \in [\lambda_{\min}, \lambda_{\max}]$. 

	\begin{theorem}\label{thm:gen-int-reg}
		Suppose \cref{as:lips-hess} holds. 
		Let $r_{\max} = \max_{0 \le k \le K} r_k$. 
		Then, linear interpolation $\hat{x}(\cdot), \hat{\lambda}(\cdot)$ of sequence $\lrrr{(x_k, \lambda_k)}_{k=0}^K$ generated by \cref{eq:linear-int} has the following computational guarantee for all $t \in [t_0, t_K]$: 
		\begin{equation*}
		    \vvt{\nabla F_{\hat{\lambda}(t)} (\hat{x}(t))} \le r_{\tn{max}} + \frac{L}{8} \cdot \max_{k \in [K-1]} \lrrr{(1+\lambda_k) \vvta{x_{k+1}-x_k}^2 + 2h \vta{\xi(\lambda_k)} \vvta{x_{k+1}-x_k}}. 
		\end{equation*}
	\end{theorem}
	\begin{proof}
		First suppose $t \in [t_k, t_{k+1}]$. 
		For simplicity, we define $x := \hat{x}(t)$, $\lambda := \hat{\lambda}(t)$, $\delta_1 := \nabla f(x_k) + \lambda_k \nabla \Omega(x_k)$, $\delta_2 := \nabla f(x_{k+1}) + \lambda_{k+1} \nabla \Omega(x_{k+1})$. 
		By triangle inequality and \cref{it:tay-exp-1} in \cref{lm:tay-exp}, it holds that 
		\begin{align*}
		    \vvt{\alpha \nabla f(x_k) + (1-\alpha) \nabla f(x_{k+1}) - \nabla f(x)} & \le \frac{\alpha(1-\alpha)L}{2} \vvt{x_k - x_{k+1}}^2. 
		\end{align*}
		Also, by applying similar trick on $\nabla \Omega(\cdot)$, it holds that 
		\begin{align*}
			&\vvt{\lambda \nabla \Omega(x) - \alpha \lambda_k \nabla \Omega(x_k) - (1-\alpha) \lambda_{k+1} \nabla \Omega(x_{k+1})} \\
			\le &\vvt{\lambda (\alpha \nabla \Omega(x_k) + (1-\alpha) \nabla \Omega(x_{k+1}) - \nabla \Omega(x))} \\ 
			& + \vvt{\alpha (\lambda - \lambda_k) \nabla \Omega(x_k) + (1-\alpha) (\lambda - \lambda_{k+1}) \nabla \Omega(x_{k+1})} \\
			\le &\frac{\lambda L}{2} \alpha (1-\alpha) \vvt{x_{k+1}-x_k}^2 + \vvt{\alpha(1-\alpha)(\lambda_{k+1}-\lambda_{k}) (\nabla \Omega(x_{k+1}) - \nabla \Omega(x_k))} \\
			\le &\frac{\lambda L}{8} \vvt{x_{k+1}-x_k}^2 + \frac{h \vt{\xi(\lambda_k)} L}{4} \vvt{x_{k+1}-x_k}. 
		\end{align*}
		Combine the above two inequality and apply triangle inequality, we have
		\begin{align*}
			&\vvt{\nabla f(x) + \lambda \nabla \Omega(x) - \alpha \delta_1 - (1-\alpha) \delta_2} \\
			\le & \frac{L}{8} \vvt{x_{k+1}-x_k}^2 + \frac{\lambda L}{8} \vvt{x_{k+1}-x_k}^2 + \frac{\vt{\xi(\lambda_k)} L h}{4} \vvt{x_{k+1}-x_k}. 
		\end{align*}
	\end{proof} 
		
	\Cref{thm:gen-int-reg} provides the computational guarantee of linear interpolation of the sequence $\lrrr{(x_k, \lambda_k}_{k=0}^K$ generated by the update scheme \cref{eq:gen-euler1}. 
	Observing that $\vvt{x_{k+1} - x_k}$ is of the order $\mathcal{O}(h)$, we see that the additional error incurred by linear interpolation is of the order $\mathcal{O}(h^2)$. 
	Together with the $\mathcal{O}(h)$ accuracy from the update scheme \cref{eq:gen-euler1}, we are able to provide a computational guarantee on the accuracy of the approximate path generated by \cref{alg:euler}. 
	In the following part, we will provide a uniform bound on $\vvta{\nabla F_{\hat{\lambda}(t)} (\hat{x}(t))}$ for a family of $\lambda(t)$ and discretization. 

	\subsection{Computational Guarantee for the Exponential Decaying Parameter Sequence}
	Under the $\lambda(t) = \lambda_{\max} \cdot \exp(-t)$ scenario, we have $\xi(\lambda) = -\lambda$ and $\xi(\cdot)$ satisfies the assumption in \cref{prop:global-rk}. 
	Therefore, we extend the result of \cref{thm:gen-int-reg} and provide an explicit uniform bound of the path accuracy. 
	
	\begin{proposition}\label{prop:gen-int-reg-exp}
	    Suppose \cref{as:lips-hess} holds, and the step-size $h$ satisfies the condition in \cref{eq:global-h-2}. 
	    Let $\hat{x}(\cdot): [0, T] \to \bbR^p$ and $\hat{\lambda}(\cdot): [0, T] \to [\lambda_{\min}, \lambda_{\max}]$ denote the approximate solution path generated by \cref{alg:euler}. 
	    Then, we have the following computational guarantee for all $\lambda \in [\lambda_{\min}, \lambda_{\max}]$: 
		\begin{equation}\label{eq:acc-path-uniform}
		    \vvt{\nabla F_\lambda (\hat{x}(\lambda))} \le \vvt{\nabla F_{\lambda_{\max}}(x_0)} + 2h \tau L (f(x_0)-f^*) + \frac{h^2 L}{8} \cdot (\tau G + 1)^2. 
		\end{equation}
	\end{proposition}
	\begin{proof}
	    First we extend the result in \cref{prop:global-rk}, and it holds that 
	    \begin{equation*}
	        r_{\max} \le \max_{k \in [K]} r_k \le r_0 + 2 h L (f(x_0) - f^*) \cdot \tau. 
	    \end{equation*}
	    Also, for the result in \cref{thm:gen-int-reg}, we further have 
	    \begin{align*}
	        & \frac{L}{8} \cdot \max_{k \in [K-1]} \lrrr{(1+\lambda_k) \vvta{x_{k+1}-x_k}^2 + 2h \vta{\xi(\lambda_k)} \vvta{x_{k+1}-x_k}} \\ 
	        \le & \frac{h^2 L}{8} \cdot \max_{k \in [K-1]} \lrrr{\frac{(1 + \lambda_k) G^2}{(\mu + \lambda_k \sigma)^2} + \frac{2 \lambda_k G}{\mu + \lambda_k \sigma}} \le \frac{h^2 L}{8} \cdot (\tau G + 1)^2. 
	    \end{align*}
	    By combining the above two inequalities, \cref{prop:global-rk}, and \cref{thm:gen-int-reg}, we obtain \cref{eq:acc-path-uniform}. 
	\end{proof}

	In \cref{alg:euler}, the sequence $\lrrr{\lambda_k}_{k=0}^{K}$ is given by $\lambda_{k+1} = (1-h)\lambda_k$, and implies $\lambda_{\min} = (1-h)^K \lambda_{\max}$. 
	Therefore, we have $h = 1 - (\tfrac{\lambda_{\min}}{\lambda_{\max}})^{1/K}$. 
	Applying this to \cref{prop:gen-int-reg-exp}, we arrive at the complexity analysis with respect to the number of iterations. 
	We formalize the complexity analysis in the following proof of \cref{thm:err-ctrl-exp}, which appears at the beginning of this section. 

	\begin{proof}[Proof of \cref{thm:err-ctrl-exp}]
		The conditions that $K \ge \max \lrrr{2T, \frac{\sqrt{LG} \tau T}{\sqrt{3} }}$ and $h = 1 - \lr{\frac{\lambda_{\min}}{\lambda_{\max}}}^{1/K} \le \frac{T}{K}$ guarantee that step-size $h$ satisfies \cref{eq:global-h-2}. 
		Also, $K \ge \frac{\tau L T (f(x_0)-f^*)}{\epsilon}$ and $K \ge \frac{(\tau G + 1) \sqrt{L} T}{\sqrt{\epsilon}}$ guarantee that $2 h \tau L (f(x_0)-f^*) \le \frac{\epsilon}{2}$ and $\frac{h^2 L}{8} \cdot (\tau G + 1)^2 \le \frac{\epsilon}{4}$. 
		Hence \cref{alg:euler} guarantees a $\epsilon$-accurate solution path. 
		\end{proof}

	Recall that in the assumption of \cref{thm:err-ctrl-exp}, it requires a good initialization $x_0$ satisfying $\vvt{\nabla F_{\lambda_{\max}}(x_0)} \le \frac{\epsilon}{2}$. 
	In practical cases, we can either implement a specific convex optimization algorithm to solve a $x_0$ satisfying the initial condition or use the initialization suggested in the following lemma. 
	Here we suggest one choice of initialization with computational guarantee when the function $\Omega(\cdot)$ is structured, \ie the optimal solution to minimize $\Omega(\cdot)$ is easy to compute. 
	\newtext{
	Let $x_{\Omega} := \arg \min_{x \in \bbR^p} \Omega(x)$ and $x_0 := x_{\Omega} - (\nabla^2 f(x_{\Omega}) + \lambda_{\max} \nabla^2 \Omega(x_{\Omega}))^{-1} \nabla f(x_{\Omega})$. 
	Intuitively, the initialization $x_0$ is a Newton step from the optimizer of $\Omega(\cdot)$, and we can show that $\vvt{\nabla F_{\lambda_{\max}}(x_0)} \le \frac{L (1 + \lambda_{\max}) \vvta{\nabla f(x_{\Omega})}^2}{2 (\mu + \lambda_{\max} \sigma)^2}$. 
	}
	Notice that the value of $L$ and $\vvt{\nabla f(x_{\Omega})}$ are independent of $\lambda_{\max}$. 
	Hence, when $\lambda_{\max}$ is large enough, we have $r_0 \le \frac{\epsilon}{4}$. 
	Also, since $x_{\Omega}$ is the optimal solution when $\lambda = +\infty$, the initialization can be considered as an update step of \cref{eq:gen-euler1} from $x_{\Omega}$.

\section{Multi-Stage Discretizations}\label{sec:multi}

    In the analysis of the previous section, \cref{alg:euler} guarantees an $\epsilon$-accurate solution path within $\mathcal{O}(\epsilon^{-1})$ calls to the gradient, Hessian oracle, and linear system solver. 
    One advantage of the main result proposed in \cref{thm:err-ctrl-exp} is that only the smooth Hessian of $f(\cdot)$ and $\Omega(\cdot)$ is required and no assumption of $\epsilon$ is required. 
    When $f(\cdot)$ and $\Omega(\cdot)$ have better properties and $\epsilon$ is relatively small, one would like an algorithm that utilizes these properties and requires fewer calls to the oracle with respect to the order of $\epsilon$. 
	Motivated by the multistage numerical methods for solving differential equations, we design several update schemes to achieve higher-order accuracy. 
	In particular, in this section we still consider the exponentially decaying parameter, that is, $\lambda(t) = e^{-t} \lambda(0)$. 

	\subsection{The Trapezoid Method}\label{sec:trapezoid}
	
	In this section, we propose and analyze the \emph{trapezoid method}, whose formal description is given in \cref{alg:trapezoid}. 
	The trapezoid method is beneficial when the functions $f(\cdot)$ and $\Omega(\cdot)$ have Lipschitz continuous third-order derivatives, and it achieves a higher-order accuracy than the implicit Euler method. 
	The accuracy of the output path of \cref{alg:trapezoid} is of order $\mathcal{O}(h^2)$ where $h$ is the step-size, or equivalently, $\mathcal{O}(K^{-2})$ where $K$ is the number of iterations. 
	Moreover, we want to mention that the algorithm does not require the oracle to higher-order derivatives, but still requires a gradient and Hessian oracle as in the Euler method. 
		
	\begin{algorithm}
		\caption{Trapezoid method for solution path}
		\label{alg:trapezoid}
		\begin{algorithmic}[1]
		    \STATE{\textbf{Input:} Initial point $x_0 \in \bbR^p$, total number of iterations $K \geq 1$}
		    \STATE{Initialize parameter $\lambda_0 \gets \lambda_{\max}$, set step-size $h \gets 1 - \sqrt{2(\frac{\lambda_{\min}}{\lambda_{\max}})^{\frac{1}{K}} - 1}$}
		    \FOR{$k = 0, \dots, K-1$}{
		        \STATE{$d_{k,1} \gets v(x_k, \lambda_k)$ and $d_{k,2} \gets v(x_k + h \cdot d_{k,1}, (1-h+h^2)\lambda_k)$}
		        \STATE{$x_{k+1} \gets x_k + h \cdot \frac{d_{k,1}+d_{k,2}}{2}$ and $\lambda_{k+1} \gets (1-h+\frac{h^2}{2}) \lambda_k$}
		    }
		    \ENDFOR
		    \RETURN{$\hat{x}(\cdot) \gets \mathcal{I}_{\textnormal{linear}}\lr{{\{(x_k, \lambda_k)\}}_{k=0}^K}$ according to linear interpolation}
		\end{algorithmic}
	\end{algorithm}
		
	We first state the main technical assumption and computational guarantees of \cref{alg:trapezoid}. 
		    
	\begin{assumption}\label{as:third-lips}
	    In addition to \cref{as:lips-hess}, we assume that the third-order directional derivative of $f(\cdot)$ and $\Omega(\cdot)$ are $L$-Lipschitz continuous and $\sigma \ge 1$. 
	\end{assumption} 
			
	\begin{theorem}\label{thm:err-ctrl-trapezoid}
	    Suppose \cref{as:third-lips} holds, let $\epsilon > 0$ be desired accuracy, let $\tilde{\mu} := \mu + \lambda_{\min} \sigma$, suppose that the initial point $x_0$ satisfies $\vvt{\nabla F_{\lambda_{\min}}(x_0)} \le \frac{\epsilon}{2} \leq \tilde{\mu}$, let $T := 1.1 \log(\lambda_{\max} / \lambda_{\min})$, and let 
    	\begin{equation*}
	        K_{\tn{tr}} := \left\lceil \max \lrrr{10T, \frac{8LT(1+G)}{\tilde{\mu}}, \frac{6 L^{1/2} (1+G)^{3/2} T}{\epsilon^{1/2}}, \frac{5 \tau^{2/3} L (1+G)^{4/3}T}{\epsilon^{1/3}}} \right\rceil. 
        \end{equation*}
		If the total number of iterations $K$ satisfies $K \geq K_{\tn{tr}}$, then \cref{alg:trapezoid} returns an $\epsilon$-accurate solution path.
	\end{theorem}
			
	The result in \cref{thm:err-ctrl-trapezoid} shows that we improve the total complexity to $\mathcal{O}(\frac{1}{\sqrt{\epsilon}})$, which is better than the $\mathcal{O}(\frac{1}{\epsilon})$ complexity of the Euler method and the best known results for the grid search method (see \cref{thm:err-ctrl-exp,rmk:grid-search}). 
	Similarly to the previous complexity analysis of the semi-implicit Euler method, the analysis of the trapezoid method consists of two parts: we first present the computational guarantee of the trapezoid update scheme, which is defined as 
	\begin{equation}\label{eq:trapezoid}
	    (x_{k+1}, \lambda_{k+1}) \gets T(x_k, \lambda_k) := \lr{x_k + h \cdot \tfrac{d_{k,1}+d_{k,2}}{2}, (1-h+\tfrac{h^2}{2}) \lambda_k}, 
	\end{equation}
	where $d_{k, 1} = v(x_k, \lambda_k)$, and $d_{k, 2} = v(x_k + h \cdot d_{k,1}, (1-h+h^2) \lambda_k)$. 
	
	\begin{lemma}\label{lmm:trapezoid-local}
		Suppose \cref{as:third-lips} holds, $r_k = \vvt{\nabla F_{\lambda_k}(x_k)} \le \tilde{\mu}$, and the next iterate $(x_{k+1}, \lambda_{k+1})$ is given by $(x_{k+1}, \lambda_{k+1}) = T(x_k, \lambda_k)$ defined in \cref{eq:trapezoid}. 
		Then, it holds that $(1 + \lambda_k) \vvta{x_k - x_{k+1}} \le 3 h (1+G)$,  and 
		\begin{equation}\label{eq:third-err}
			r_{k+1} := \vvt{\nabla F_{\lambda_{k+1}}(x_k)} \le \frac{\lambda_{k+1}}{\lambda_k} \cdot r_k + h^3 \cdot 3 L (1+G)^3 + h^4 \cdot 2 L^3 \tau^2 (1+G)^4. 
		\end{equation}
	\end{lemma}
	
	We leave the proof of \cref{lmm:trapezoid-local} in the appendix because of its length and complexity. 
	In the proof, we mainly work with the directional derivatives (see \citet{nesterov2018implementable}, for details) and the accuracy of the Taylor expansion in \cref{lm:tay-exp}. 
	The result in \cref{lmm:trapezoid-local} shows that the trapezoid update scheme in \cref{eq:trapezoid} guarantees a local accumulation of $\mathcal{O}(h^3)$. 
	Moreover, we can derive an $\mathcal{O}(h^2)$ uniform upper bound on the accuracy of all nearly optimal solutions $\lrrr{x_k}$. 
	For all other $\lambda \in [\lambda_{\min}, \lambda_{\max}]$, we implement the linear interpolation to approximate the corresponding nearly-optimal solution. 
	We then provide the formal proof of \cref{thm:err-ctrl-trapezoid}. 

	\begin{proof}[Proof of \cref{thm:err-ctrl-trapezoid}]
	    Since $h - \frac{h^2}{2} \le \frac{1}{K} \log (\frac{\lambda_{\max}}{\lambda_{\min}})$, it holds that 
	    \begin{equation*}
			h \le \min \lrrr{0.1, \frac{\tilde{\mu}}{8L(1+G)}, \frac{\epsilon^{1/2}}{3 L^{1/2} (1+G)^{3/2}}, \frac{\epsilon^{1/3}}{3 \tau^{2/3} L (1+G)^{4/3}}}. 
		\end{equation*}
		Then we show that $r_k := \vvt{\nabla F_{\lambda_k}(x_k)} \le \frac{\epsilon}{2}$ for all $k$ by induction. 
		Suppose $r_k \le \frac{\epsilon}{2}$, then by \cref{lmm:trapezoid-local}, it holds that 
		\begin{equation*}
		    r_{k+1} := \vvt{\nabla F_{\lambda_{k+1}}(x_k)} \le \frac{\lambda_{k+1}}{\lambda_k} \cdot r_k + h^3 \cdot 3 L (1+G)^3 + h^4 \cdot 2 L^3 \tau^2 (1+G)^4 \le \frac{\epsilon}{2}. 
		\end{equation*}
		Therefore, $r_k \le \frac{\epsilon}{2}$ for all $k \in \{0, \dots, K\}$. 
		Suppose $\lambda \in [\lambda_{k+1}, \lambda_k]$, and hence $\hat{x}(\lambda) = \alpha x_k + (1-\alpha) x_{k+1}$ where $\alpha = \frac{\lambda-\lambda_{k+1}}{\lambda_k-\lambda_{k+1}}$. 
		Applying the results in \cref{thm:gen-int-reg}, we have 
		\begin{align*}
			\vvt{f(\hat{x}(\lambda)) + \lambda \hat{x}(\lambda)} & \le \frac{\epsilon}{2} + \frac{L}{8} \max_{k \in [K-1]} \lrrr{(1+\lambda_k) \vvta{x_{k+1}-x_k}^2 + 2h \vta{\xi(\lambda_k)} \vvta{x_{k+1}-x_k}} \\ 
			& \le \frac{\epsilon}{2} + \frac{L}{8} \cdot \lr{9 h^2 (1+G)^2 + 6 h^2 (1+G)} \le \epsilon. 
		\end{align*}
	\end{proof}

\section{Analysis with Inexact Linear Equations Solutions and Second-Order Conjugate Gradient Variants}\label{sec:inexact}
	
    In this section, we present the complexity analysis of the aforementioned methods with the presence of inexact oracle to gradient and Hessian and/or inexact linear equations solutions, as well as variants applying the second-order conjugate gradient (SOCG) type methods. 
	We first consider the case when the gradient and Hessian oracle are inexact and/or the linear equation solver yields approximate solutions. 

	\subsection{Analysis with Inexact Oracle and/or Approximate Solver}
	
    At each iteration of the Euler, trapezoid, and Runge-Kutta method, an essential subroutine is to compute the directions $d_{k, i}$. 
    For example, in the Euler method, $d_k$ is given by formula $d_k = v(x_k, \lambda_k) = -(\nabla^2 f(x_k) + \lambda_{k+1} \nabla^2 \Omega(x_k))^{-1} \nabla f(x_k)$. 
    In most large-scale problems, the computation of the Hessian matrix and solving linear equations exactly will be the computational bottleneck.  
	Therefore, we consider the case with the presence of numerical error, which may be induced by inexact gradient and Hessian oracle, or by the linear equations solver. 
	Nevertheless, we tackle the two types of numerical error together. 
	\newtext{
	\begin{definition}
	    Suppose an exact solution $d_k$ has the form $d_k = - H_k^{-1} g_k$. For any constant $\delta \ge 0$, $\hat{d}_k$ is a $\delta$-approximate direction with respect to $d_k$ if $\vvta{H_k \hat{d}_k + g_k} \le \delta$. Furthermore, the $\delta$-approximate versions of \Cref{alg:euler} or \Cref{alg:trapezoid} are the same as the original algorithms except that they use a $\delta$-approximate direction in each update.  
	\end{definition}
	}

	\newtext{
    Compared with the original update scheme in \cref{alg:euler}, the only difference in the update scheme of the $\delta$-approximate version is that an approximate direction $\hat{d}_k$ is applied at each iteration instead of the exact direction $v(\cdot, \cdot)$. 
    We want to mention that there is no constraint on how the approximate direction $\hat{d}_k$ is generated, and in \Cref{sec:socg} we provide several efficient methods to compute the approximate direction and the corresponding complexity analysis. 
    The following two lemmas characterize the accumulation of local errors of the update scheme in the $\delta$-approximate version of \Cref{alg:euler} and \Cref{alg:trapezoid}. 
    }

	\begin{lemma}\label{lm:approx-euler}
		Suppose \cref{as:lips-hess} holds. 
		Let $\hat{d}_k$ denote an approximate solution to $v(x_k, \lambda_k)$. 
		Let $\delta_k = \lr{\nabla^2 f(x_k) + \lambda_{k+1} \nabla^2 \Omega(x_k)} \hat{d}_k + \nabla f(x_k)$, $r_k = \vvta{\nabla F_{\lambda_k}(x_k)}$, and $r_{k+1} = \vvta{\nabla F_{\lambda_{k+1}}(x_{k+1})}$. 
		Then we have
        \begin{equation}\label{eq:robust-euler-1}
		    r_{k+1} \le \frac{\lambda_{k+1}}{\lambda_k} \cdot r_k + \frac{h^2 L (1+\lambda_{k+1})}{2} \cdot \vvta{\hat{d}_k}^2 + h \vvt{\delta_k}. 
		\end{equation}
		Furthermore, if we set $\lambda_{s+1} = (1-h) \lambda_s$ and $\vvt{\delta_s} \le \delta$ for some scalar $\delta > 0$ and all $s = 0, \dots, k$, it holds that 
		\begin{equation}\label{eq:robust-euler-2}
			r_{k} \le \frac{\lambda_k}{\lambda_0} \cdot r_0 + 2 h \tau L (f(x_0)-f^*) + \delta. 
		\end{equation}
	\end{lemma}
	\begin{proof}
		First it holds that 
		\begin{align*}
		    & \lambda_{k+1} \lr{\nabla \Omega(x_k) + \nabla^2 \Omega(x_k) (x_{k+1}-x_k)} + \nabla f(x_k) + \nabla^2 f(x_k) (x_{k+1}-x_k) \\ 
		    = & \frac{\lambda_{k+1}}{\lambda_k} (\lambda_k \nabla \Omega(x_k) + \nabla f(x_k)) + h \cdot \delta_k.
		\end{align*}
		Then apply same technique as in \cref{lm:gen-err-2} we show that \cref{eq:robust-euler-1} holds. 
		Also, applying \cref{eq:robust-euler-1} to \cref{prop:gen-int-reg-exp} implies the result in \cref{eq:robust-euler-2}. 
	\end{proof}
	
	\begin{lemma}\label{lm:robust-trap}
		Suppose $r_k = \vvt{\nabla F_{\lambda_k}(x_k)} \le \tilde{\mu}$. 
		Let $\delta_{k,1}, \delta_{k,2}$ denote the residual of the approximate directions $\hat{d}_{k,1}, \hat{d}_{k,2}$ with $\vvt{\delta_{k,1}}, \vvt{\delta_{k,2}} \le \tilde{\mu}$. 
		Then, it holds that 
		\begin{equation}\label{eq:robust-trap-0}
			r_{k+1} \le \frac{\lambda_{k+1}}{\lambda_k} \cdot r_k + h^3 \cdot 3 L (2+G)^3 + h^4 \cdot 2 L^3 \tau^2 (2+G)^4 + \frac{h}{2} \vvt{\delta_{k,1}-\delta_{k,2}} + \frac{h^2}{2} \vvt{\delta_{k,1}}. 
		\end{equation}
	\end{lemma}
	\begin{proof}
		We will follow the idea in \cref{lm:bound-norm,lm:diff-ind-reg,lmm:trapezoid-local}. 
		Recall the result in \cref{lm:bound-norm}, since $\tilde{H}_1 d_1 = -\nabla f(x_1) + \delta_1$, it holds that $(1+\lambda) \vvt{d_1} \le 2(G+1+\vvt{\delta_1}/\tilde{\mu})$. 
		Also, the result in \cref{lm:diff-ind-reg} becomes 
		\begin{equation*}
			\vvt{\tilde{H}_1 (d_2 - d_1) - \nabla^2 f(x_1) (x_1-x_2) - (\tilde{H}_1 - \tilde{H}_2) d_2 + \delta_1 - \delta_2} \le \frac{L}{2} \vvt{x_1-x_2}^2,  
		\end{equation*}
		where $\tilde{H}_1 = \nabla^2 f(x_1) + \lambda_1 \nabla^2 \Omega(x_1)$ and $\tilde{H}_2 = \nabla^2 f(x_2) + \lambda_2 \nabla^2 \Omega(x_2)$. 
		Now we modify the proof of \cref{lmm:trapezoid-local} to get \cref{eq:robust-trap-0}. 
		Then the right-hand side of \cref{eq:local-trap-1} becomes $-h \delta_1$. 
		Also, the right-hand side of \cref{eq:local-trap-2} becomes $\frac{h}{2}(\delta_1-\delta_2)$ and the right-hand side of \cref{eq:local-trap-3} becomes $\frac{h^2}{2} \lambda \nabla^2 \Omega(x) d_1 + \frac{h^2}{2} \delta_1$. 
	\end{proof}
			
	The following two corollaries provide the complexity analysis of the $\delta$-approximate version of \Cref{alg:euler} and \Cref{alg:trapezoid}. 
			
	\begin{corollary}\label{crl:err-ctrl-euler}
		Suppose that \cref{as:lips-hess} holds, and suppose that the initial point $x_0$ satisfies $\vvt{\nabla F_{\lambda_{\max}}(x_0)} \le \frac{\epsilon}{4}$. Let $T := \log(\lambda_{\max}/\lambda_{\min})$, let $\epsilon \in (0, \mu + \lambda_{\min} \sigma]$ be the desired accuracy and let \begin{equation*}
	        K_{\tn{E, approx}} := \left\lceil \max \lrrr{2T, \frac{\sqrt{LG} \tau T}{\sqrt{3} }, \frac{8 (f(x_0) - f^\ast) \tau L T}{\epsilon}, \frac{4 \sqrt{L} (\tau (G + \epsilon) + 1) T}{\sqrt{\epsilon}}} \right\rceil. 
    	\end{equation*}
        If the number of iterations $K$ satisfies $K \geq K_{\tn{E, approx}}$ and approximate directions $\hat{d}_k$ are all $\frac{\epsilon}{4}$-approximate, the $\delta$-approximate version of \Cref{alg:euler} returns an $\epsilon$-accurate solution path.
	\end{corollary}
	\begin{proof}
	    Since the step-size $h = 1 - (\frac{\lambda_{\min}}{\lambda_{\max}})^{\frac{1}{K}} \le \frac{T}{K}$, combining \cref{eq:robust-euler-2}, we have $r_k \le \frac{3 \epsilon}{4}$, for all $k \in \{0, \dots, K\}$. 
	    For linear interpolation error, we have
	    \begin{align*}
	        & \frac{L}{8} \cdot \max_{k \in [K-1]} \{(1+\lambda_k) \vvta{x_{k+1}-x_k}^2 + 2h \vta{\xi(\lambda_k)} \vvta{x_{k+1}-x_k}\} \\ 
	        \le & \frac{h^2 L}{2} \cdot \max_{k \in [K-1]} \{(1+\lambda_k) (\vvta{d_k}^2 + \vvta{d_k - \hat{d}_k}^2) + \vta{\xi(\lambda_k)} (\vvta{d_k} + \vvta{d_k - \hat{d}_k})\} \\ 
	        \le &  \frac{h^2 L}{2} \cdot (\tau (G + \epsilon) + 1)^2 \le \frac{\epsilon}{4}. 
	    \end{align*}
	    Applying the above inequality to \cref{thm:gen-int-reg} completes the proof. 
	\end{proof}

	\begin{corollary}\label{crl:err-ctrl-trapezoid}
	    Suppose \cref{as:third-lips} holds, let $\tilde{\mu} := \mu + \lambda_{\min} \sigma$, let $\epsilon \in (0, \tilde{\mu}]$ be the desired accuracy, suppose that the initial point $x_0$ satisfies $\vvt{\nabla F_{\lambda_{\max}}(x_0)} \le \frac{\epsilon}{4}$, let $T := 1.1\log(\lambda_{\max}/\lambda_{\min})$, and let
    	\begin{equation*}
			 K_{\tn{tr, approx}} := \left\lceil \max \lrrr{10T, \frac{8LT(2+G)}{\tilde{\mu}}, \frac{6 L^{1/2} (2+G)^{3/2} T}{\epsilon^{1/2}}, \frac{6 L \tau^{2/3} (2+G)^{4/3} T}{\epsilon^{1/3}}} \right\rceil. 
        \end{equation*}
	    If the number of iterations $K$ satisfies $K \geq K_{\tn{tr, approx}}$ and all approximate directions $\hat{d}_{k,1}, \hat{d}_{k,2}$ are $\frac{\epsilon}{4}$-approximate, the $\delta$-approximate version of \Cref{alg:trapezoid} returns an $\epsilon$-accurate solution path.
	\end{corollary}
	\begin{proof}
	    Using \cref{eq:robust-trap-0} for $r_{k+1}, \dots, r_1$, we have 
	    \begin{align*}
	        r_{k+1} & \le r_0 + \frac{3 h^3 L (2+G)^3 + 2 h^4 L^3 \tau^2 (2+G)^4 + \frac{h}{2} \vvt{\delta_{k,1}-\delta_{k,2}} + \frac{h^2}{2} \vvt{\delta_{k,1}}}{1 - \lr{1 - h + \frac{h^2}{2}}} \\ 
	        & \le \frac{\epsilon}{4} + \frac{3 h^2 L (2+G)^3 + 2 h^3 L^3 \tau^2 (2+G)^4 + \frac{\epsilon}{4} + \frac{h \epsilon}{8}}{1 - \frac{h}{2}} \le \frac{3 \epsilon}{4}. 
	    \end{align*}
        Also, since $\vvta{x_{k+1} - x_k} \le h (\vvta{d_k} + \frac{\vvta{\hat{d}_{k, 1}} + \vvta{\hat{d}_{k, 2}}}{2 (\mu + \lambda_k \sigma)})$, it holds that 
		\begin{align*}
			\vvt{f(\hat{x}(\lambda)) + \lambda \hat{x}(\lambda)} & \le \frac{3 \epsilon}{4} + \frac{L}{8} \cdot \max_{k \in [K-1]} \lrrr{(1+\lambda_k) \vvta{x_{k+1}-x_k}^2 + 2h \lambda_k \vvta{x_{k+1}-x_k}} \\ 
			& \le \frac{3 \epsilon}{4} + \frac{L}{4} \cdot \lr{h^2 \lr{9 (1+G)^2 + 2 \tau \epsilon} + h^2 (3 (1+G) + \epsilon)} \le \epsilon. 
		\end{align*}
	\end{proof}

    \Cref{crl:err-ctrl-euler,crl:err-ctrl-trapezoid} provide the complexity analysis of the Euler method and the trapezoid method with the presence of approximate directions. 
    We can see that when the residuals of approximate directions have a uniform upper bound over all iterations, the $\delta$-approximate versions of  \cref{alg:euler,alg:trapezoid} have complexity of the same order as the original algorithms, respectively, which require exact directions. 
    Moreover, \cref{crl:err-ctrl-euler,crl:err-ctrl-trapezoid} only require $\epsilon$-approximate directions but no assumptions about how the approximate directions are generated. 
    It provides flexibility in the choice of an approximate oracle to compute approximate directions.

	\subsection{Second-Order Conjugate Gradient Variants}\label{sec:socg}
    Following the complexity analysis, we apply the conjugate gradient method to solve the sub-problem, \ie to compute a $\delta$-approximate direction of $v(\cdot, \cdot)$. 
    To measure the efficiency of second-order conjugate gradient type algorithms, we measure the computational complexity by the total number of calls to both gradient and Hessian-vector product oracles. 
	
	Now we apply the conjugate gradient method as an approximate oracle to compute the approximate direction $\hat{d}_k$ at each iteration. 
	We use the $\delta$-approximate version of the Euler method as an example. 
    At iteration $k$, the algorithm requires an approximate solution $\hat{d}_k$ satisfying $\Vert{H_k \hat{d}_k + g_k}\Vert_2 \le \delta$ where $H_k := \nabla^2 f(x_k) + \lambda_{k+1} \nabla^2 \Omega(x_k)$ and $g_k := \nabla f(x_k)$. 
    We apply the conjugate gradient to approximately solve the equation $H_k \hat{d}_k + g_k = 0$ and set the initial guess to be the approximate direction $\hat{d}_{k-1}$ in the last iteration. 
    Herein we provide the complexity analysis of Euler-CG method and trapezoid-CG method. 
    
    \begin{theorem}\label{thm:err-ctrl-cg}
    \newtext{
        Under the same conditions as \Cref{thm:err-ctrl-exp}, let $\delta \gets \epsilon / 4$, and let $K_{\tn{E-CG}} \sim \tilde{\mathcal{O}} (L^{3/2} \tau^{3/2} (f(x_0)-f^*) T / \epsilon)$. 
	    If the total number of iterations $K$ satisfies $K \geq K_{\tn{E-CG}}$, then the Euler-CG method returns an $\epsilon$-accurate solution path. 
	    Also, suppose that the assumption in \Cref{thm:err-ctrl-trapezoid} holds, let $\delta \gets \epsilon / 4$, and let $K_{\tn{tr-CG}} \sim \tilde{\mathcal{O}} (L \tau^{1/2} (2+G)^{3/2} T / \epsilon^{1/2})$. 
	    If the total number of iterations $K$ satisfies $K \geq K_{\tn{tr-CG}}$, then the trapezoid-CG method returns an $\epsilon$-accurate solution path. 
    }
    \end{theorem}

    \begin{proof}[Proof of \cref{thm:err-ctrl-cg}]
        Here we only need to consider the inner complexity, \ie the number of iterations required to compute the approximate direction. 
        For the Euler-CG algorithm, let $\lrrr{y_{k,s}}$ denote the sequence generated by the conjugate gradient method with $y_{k,0} = \hat{d}_{k-1}$, let $H_k = \nabla^2 f(x_k) + \lambda_{k+1} \nabla^2 \Omega(x_k)$, and let $g_k = \nabla f(x_k)$. 
        The existing results of the conjugate gradient method \cite{bertsekas1997nonlinear} guarantee that $\vvt{H_k y_{k,s} + g_k}_2 \le 2 \sqrt{\kappa_k} (1-\frac{2}{\sqrt{\kappa_k}+1})^s \vvt{H_k y_{k,0} + g_k}_2$, where $\kappa_k$ is the condition number of $H_k$ and $\kappa_k \le \frac{(1 + \lambda_k) L}{\mu + \lambda_k \sigma} \le \tau L$. 
        Since the initialization $y_0 = \hat{d}_{k-1}$ is the approximate direction at the last iteration, we have $\vvta{H_{k-1} \hat{d}_{k-1} + g_{k-1}} \le \frac{\epsilon}{4}$. 
        Then the initialization guarantees that 
	    \begin{align*}
		    \vvta{H_k \hat{d}_{k-1} + g_k}_2 & \le \vvta{H_{k-1} \hat{d}_{k-1} + g_{k-1}}_2 + \vvta{H_k \hat{d}_{k-1} + g_k - H_{k-1} \hat{d}_{k-1} - g_{k-1}} \\
    	    & \le \frac{\epsilon}{4} + h L (1 + \lambda_k) \lr{\vvta{\hat{d}_{k-1}}^2 + \vvta{\hat{d}_{k-1}}} \le \frac{\epsilon}{4} + 2h L (2+G)^2.  
        \end{align*}
        Applying $h \sim \mathcal{O}(\frac{\epsilon}{(f(x_0) - f^*) \tau L})$, the inner complexity $N_k$ has an upper bound $N_k \le \frac{\sqrt{\kappa_k} + 1}{2} \log ( \frac{2 \sqrt{\kappa_k} \vvta{H_k \hat{d}_{k-1} + g_k}}{\epsilon / 4 }) \sim \tilde{\mathcal{O}}(\sqrt{\tau L})$.  
        Therefore, the total computation complexity of the Euler-CG method with the approximate oracle of the conjugate gradient to compute an $\epsilon$-accurate solution path is $\tilde{\mathcal{O}} (\frac{L^{3/2} \tau^{3/2} (f(x_0)-f^*) T}{\epsilon})$. 
        For the trapezoid-CG algorithm, let $x_k' = x_k + h \hat{d}_{k,1}$, $\lambda_k' = (1-h+h^2) \lambda_k$, $H_{k,1} := \nabla^2 f(x_k) + \lambda_k \nabla \Omega(x_k), g_{k,1} := \nabla f(x_k)$, $H_{k,2} := \nabla^2 f(x_k') + \lambda_k' \nabla \Omega(x_k')$, and $g_{k,2} := \nabla f(x_k')$. 
        At iteration $k$, the trapezoid-CG method requires to compute the $\frac{\epsilon}{4}$ directions $\hat{d}_{k,1}$ and $\hat{d}_{k,2}$. 
        In the conjugate gradient subroutine, we use $\hat{d}_{k-1, 1}$ as a warm start to compute $\hat{d}_{k, 1}$, and likewise use $\hat{d}_{k, 1}$ to warm start $\hat{d}_{k, 2}$.
        Similarly, we have $\vvta{H_{k, 1} \hat{d}_{k - 1, 1} + g_{k, 1}} \le \frac{\epsilon}{4} + 2h L (2+G)^2$ and $\vvta{H_{k, 2} \hat{d}_{k, 1} + g_{k, 2}} \le \frac{\epsilon}{4} + 2h L (2+G)^2$. 
        Applying $h \sim \mathcal{O}(\frac{\epsilon^{1/2}}{L^{1/2} (2+G)^{3/2}})$, the inner complexity $N_{k, 1}$ and $N_{k, 2}$ have upper bounds $N_{k, 1}, N_{k, 2} \sim \tilde{\mathcal{O}}(\sqrt{\tau L})$.  
        Therefore, the total computational complexity of the trapezoid-CG method with the conjugate gradient approximate oracle to compute an $\epsilon$-accurate solution path is $\tilde{\mathcal{O}} (\frac{L \tau^{1/2} (2+G)^{3/2} T}{\epsilon^{1/2}})$. 
    \end{proof}
    
    Recalling the complexity results in \cref{rmk:grid-search}, we notice that when we use Nesterov's accelerated gradient method as the subproblem solver, the total complexity will be $\frac{(\tau L)^{3/2} G T \log 2}{\epsilon}$. 
    We can see that the total complexity of the $\delta$-approximate version of \Cref{alg:trapezoid} has order $\mathcal{O}(\frac{1}{\sqrt{\epsilon}})$ and is better than the best known result for the grid search method. 
    Again, the complexity of the trapezoid method is better than that of the Euler method, since it exploits the higher-order smoothness of the functions $f(\cdot)$ and $\Omega(\cdot)$.

\section{Computational Experiments and Results}\label{sec:experiments}
	
	In this section, we present computational results of numerical experiments in which we implement different versions of discretization schemes and interpolation methods to compute the approximate solution path. 
    To compare with our methods, we introduce two approaches based on grid search (see, e.g., \citet{giesen2012approximating,ndiaye2019safe}).
    For the subproblem solver in the grid search methods, we use the warm-started exact Newton method and Nesterov's accelerated gradient method to compare with our exact methods and SOCG variants, respectively.
    We focus on the following 8 versions of the update schemes, where ``CG'' stands for conjugate gradient. In all cases, we fix an accuracy level $\epsilon > 0$ and set the step/grid-size parameters accordingly.
    \begin{itemize}
        \item Euler, Euler-CG: \cref{alg:euler}, and its $\delta$-approximate version using CG as the sub-problem approximate oracle. 
        \item Trapezoid, Trapezoid-CG: \cref{alg:trapezoid}, and its $\delta$-approximate version using CG as the approximate subproblem oracle. 
        \item Runge-Kutta, Runge Kutta-CG: an extension of our algorithm that uses the Runge-Kutta update scheme, and its $\delta$-approximate version using CG as the subproblem approximate oracle. 
        \item Grid Search-Newton/AGD: grid search method with either the exact Newton or Nesterov's accelerated gradient method as the sub-problem solver. 
    \end{itemize}
    
    \subsection{Logistic Regression}
    Herein we examine the empirical behavior of each of the previously presented methods on logistic regression problems using the breast cancer dataset from \citet{Dua:2019} ($32$ features and $569$ observations) and the leukemia dataset from \citet{golub1999molecular} ($7129$ features and $72$ observations). 
	In particular, let $\lrrr{(a_i, b_i)}_{i=1}^n$ denote a training set of features $a_i \in \bbR^p$ and labels $b_i \in \{-1, +1\}$ and define the sets of positive and negative examples by $S_+ := \lrrr{i \in [n]: b_i = 1}$ and $S_- := \lrrr{i \in [n]: b_i = -1}$. We examine two logistic regression variants: {\em (i)} regularized logistic regression with $f(x) = \frac1n \sum_{i=1}^n \log(1+e^{-b_i a_i^T x})$ and $\Omega(x) = \frac12 \vvt{x}^2$, 
	where $\lambda_{\min} = 10^{-4}$, $\lambda_{\max} = 10^4$, 
	and {\em (ii)} re-weighted logistic regression with $f(x) = \frac{1}{|S_+|} \sum_{i \in S_+} \log(1+e^{-b_i a_i^T x})$ and $\Omega(x) = \frac{1}{|S_-|} \sum_{i \in S_-} \log(1+e^{-b_i a_i^T x})$, 
	where $\lambda_{\min} = 10^{-1}$, $\lambda_{\max} = 10$. 
    The initialization $x_0$ which we apply in each method is the same and is a very nearly-optimal solution to the problem such that $\vvt{\nabla F_{\lambda_{\max}}(x_0)} \approx 10^{-15}$. 
    Note that in general it is difficult to compute the exact path accuracy of an approximate solution path $\hat{x}(\cdot): [\lambda_{\min}, \lambda_{\max}] \to \bbR^p$, namely $A(\hat{x}) := \max_{\lambda \in \Lambda} \vvt{\nabla F_{\lambda} (\hat{x}(\lambda))}$, where $\Lambda = [\lambda_{\min}, \lambda_{\max}]$. 
    We examine the approximate path accuracy in lieu of the exact computation of the path accuracy, namely, we consider $\hat{A}(\hat{x}) := \max_{\lambda \in \hat{\Lambda}} \vvt{\nabla F_{\lambda} (\hat{x}(\lambda))}$ where $\hat{\Lambda} = \{ \lambda_0, \frac{\lambda_0 + \lambda_1}{2}, \lambda_1, \frac{\lambda_1 + \lambda_2}{2}, \lambda_2, \dots, \lambda_K \}$, 
    since the theoretical largest interpolation error occurs at the midpoint of two breakpoints. 
    In \cref{fig:breast-cancer-exact}, we vary the desired accuracy parameter $\epsilon$ and plot the number of Hessian oracle computations required by \cref{alg:euler}, \cref{alg:trapezoid}, the algorithm with the Runge-Kutta update scheme, and the grid search method. 
    The theoretical evaluation numbers are derived from the complexity analysis, and the observed number of evaluations of each method is determined according to a ``doubling trick,'' whereby for each value of $K$ we calculate the observed accuracy along the path by interpolation, and if the observed accuracy is too large, then we double the value of $K$ until it is below $\epsilon$. 
    The numerical results match the asymptotic order and intuition, as well as the superior performance of the trapezoid and Runge-Kutta methods due to the higher-order smoothness of the loss and regularization function. 
    In part {\em (a)} of \cref{fig:breast-cancer-exact}, we notice that the theoretical complexity is higher than the practical one, since it is more conservative. 
    Therefore, we will stick to the ``doubling method'' in the following experiments and compare the practical performance of each method.

    \begin{figure}[tbhp]
		\centering
		\hfill
		\begin{subfloat}[Regularized logistic regression.]{\scalebox{0.40}{\includegraphics{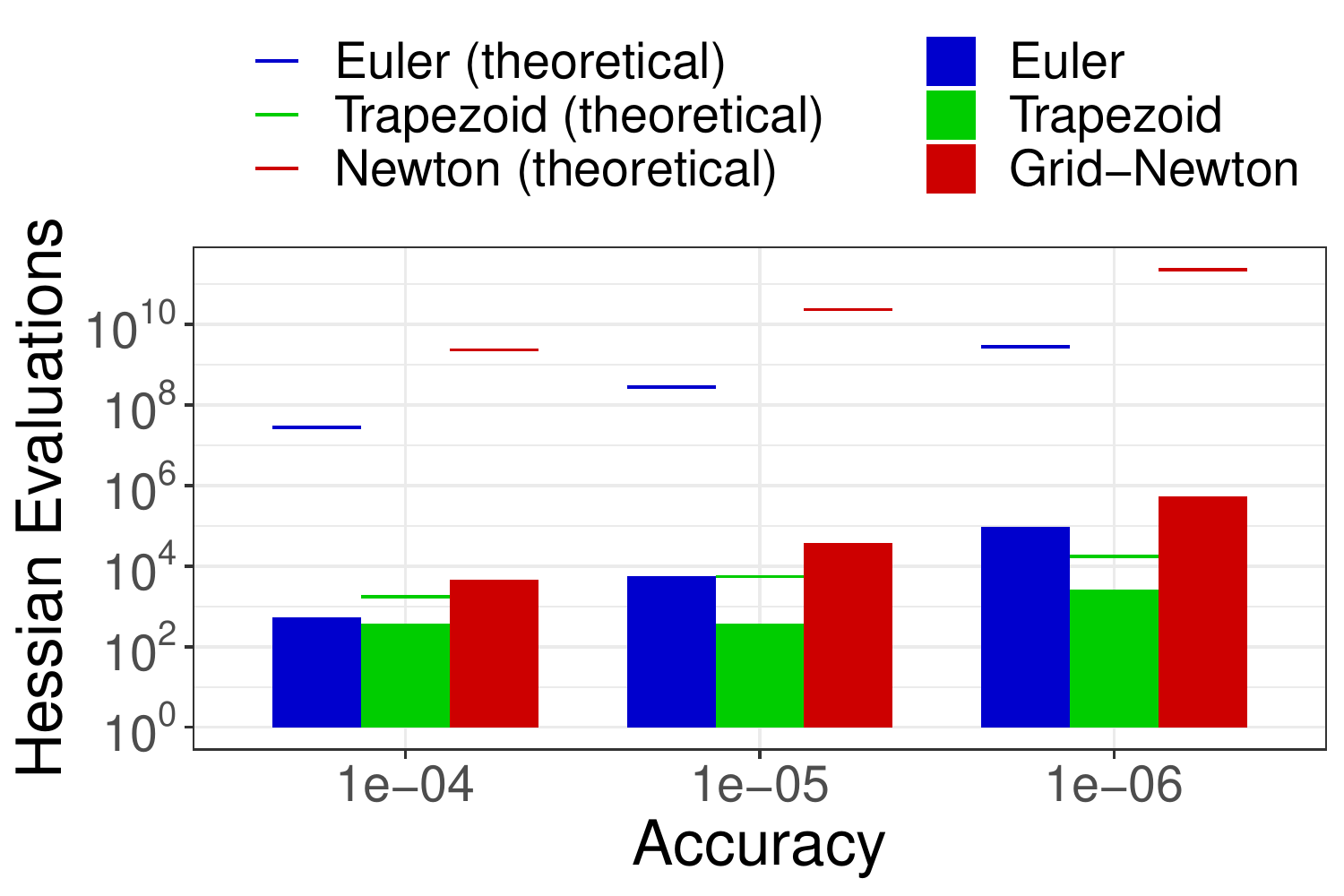}}}
		\end{subfloat}
		\hfill
		\begin{subfloat}[Re-weighted logistic regression.]{\scalebox{0.40}{\includegraphics{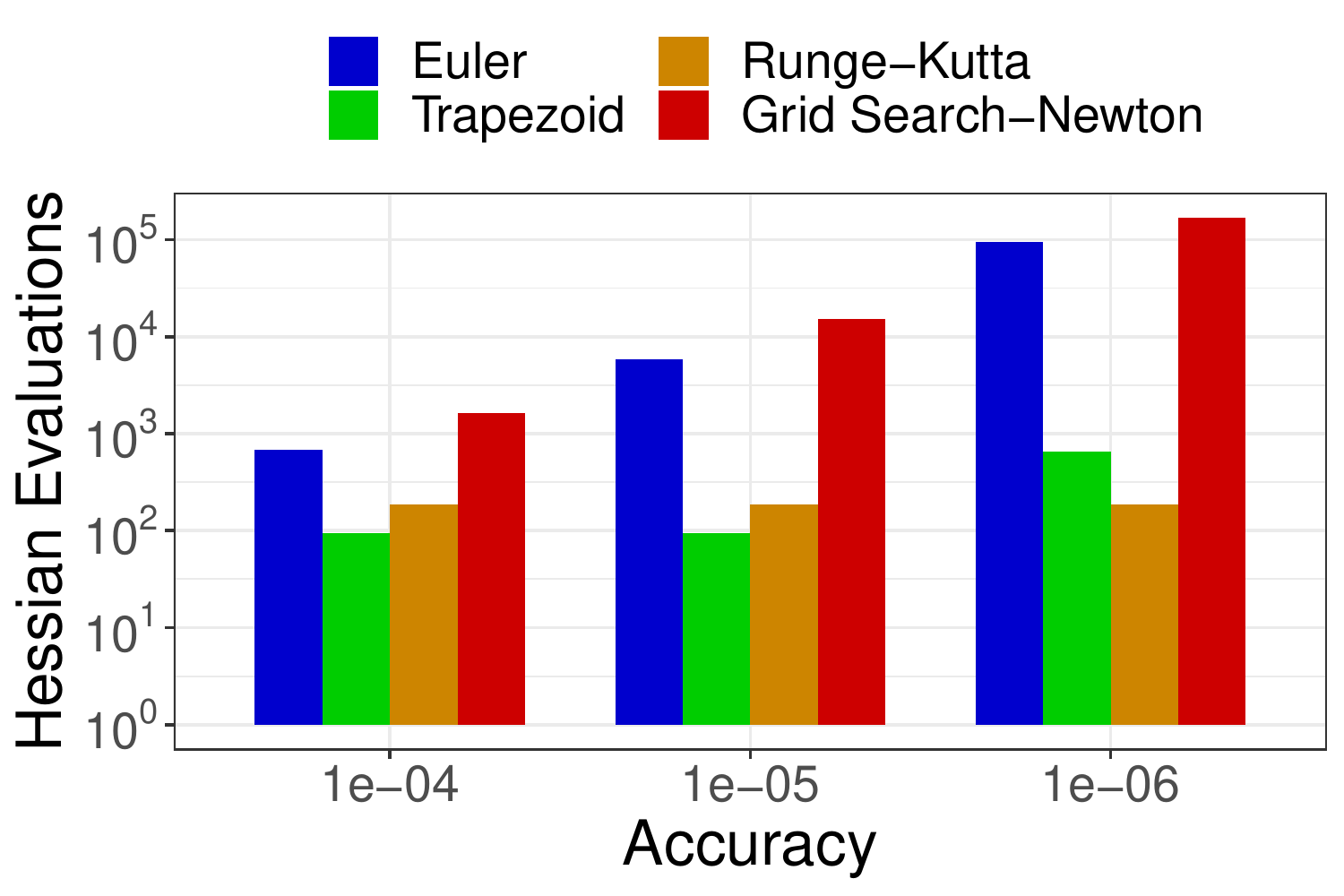}}}
		\end{subfloat}
		\hfill
        \caption{Exact methods on the breast cancer data with $n = 569$ observations and $p = 30$ features.}
		\label{fig:breast-cancer-exact}
	\end{figure}
	
	\Cref{fig:leukemia-socg} summarizes the performance of the three aforementioned SOCG methods, and the grid search method with Nesterov's accelerated gradient method. 
    For SOCG methods, we record the total number of both gradient evaluation and Hessian-vector products, which is our measure of computational complexity. 
	We implement these methods on the leukemia dataset and the regularized logistic regression problem. 
	From \cref{fig:leukemia-socg}, we can see that the numerical results match our theoretical asymptotic bounds. 
	In part {\em (b)}, we provide the CPU time to compute the approximate solutions of both SOCG methods and exact second-order methods (in the more transparent bars). 
	We comment that for the exact methods we only run the case when the desired accuracy equals $10^{-4}$, since the grid search method already takes about 15 hours to finish. 
	The dominance of SOCG methods in these cases illustrates the benefit and ability of SOCG methods to deal with large-scale problems.

    \begin{figure}[tbhp]
		\centering
		\hfill
		\begin{subfloat}[Computation complexity.]{\scalebox{0.40}{\includegraphics{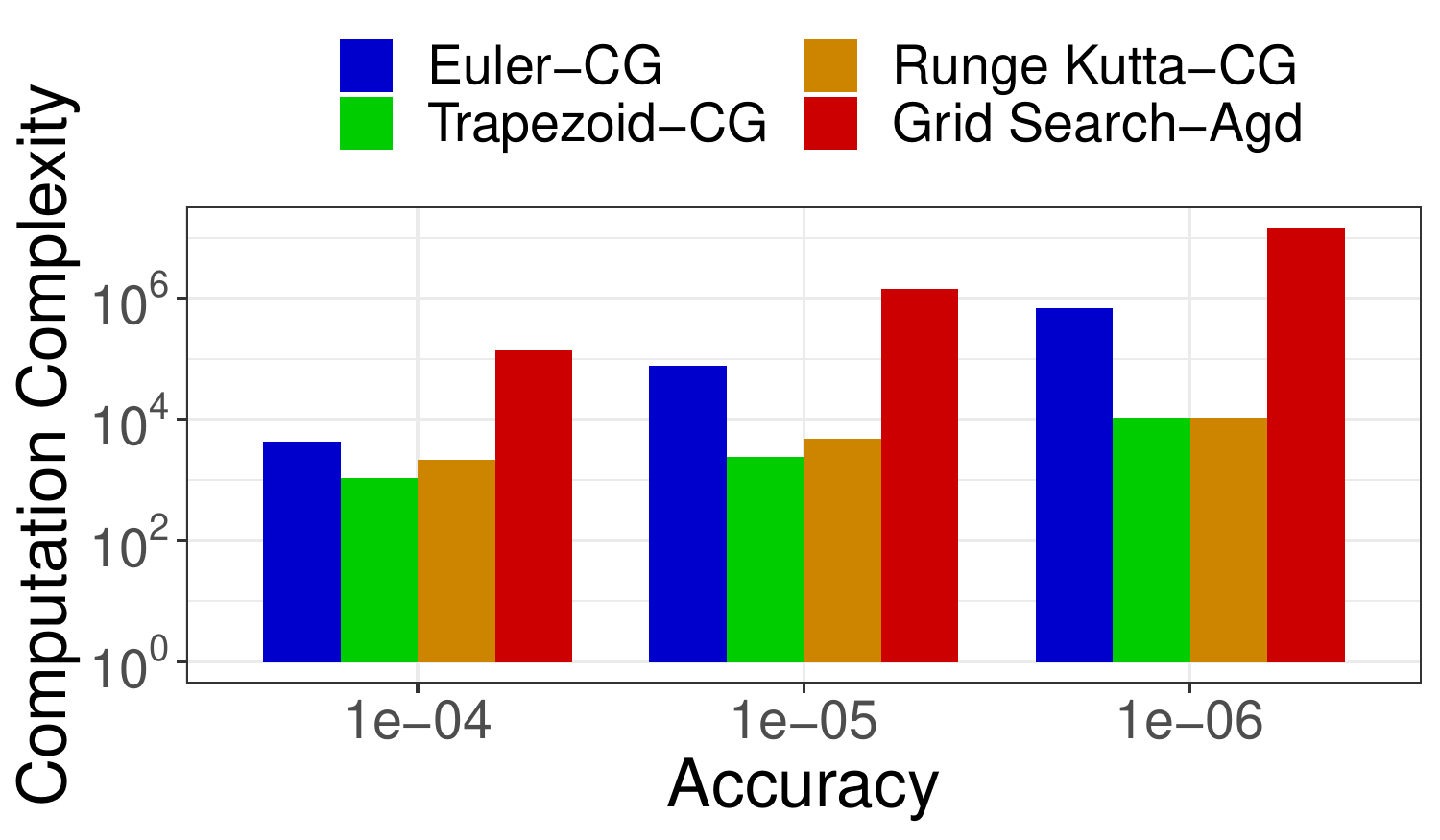}}}
		\end{subfloat}
		\hfill
		\begin{subfloat}[CPU time.]{\scalebox{0.40}{\includegraphics{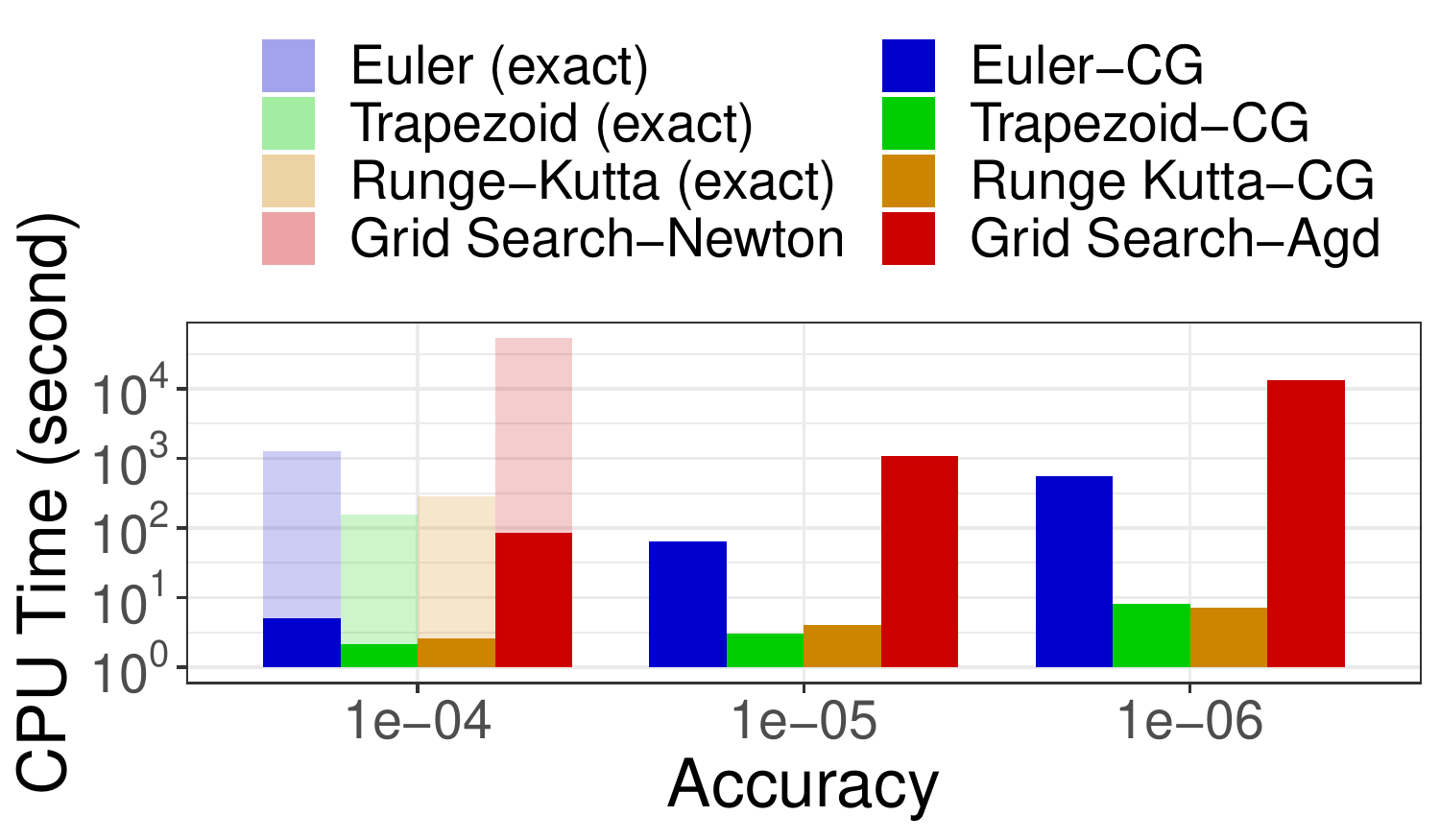}}}
		\end{subfloat}
		\hfill
        \caption{Second-order conjugate gradient methods on regularized logistic regression on leukemia data with $n = 72$ observations and $p = 7129$ features.}
		\label{fig:leukemia-socg}
	\end{figure}
    
    \subsection{Moment Matching Problem}
    
	Here we consider the moment matching problem with entropy regularization. 
    Suppose that a discrete random variable $Z$ has sample space $\{w_1, \dots, w_{p+1}\}$ and probability distribution $x = (x_{(1)}, \dots, x_{(p+1)})$. 
    The goal is to match the empirical first $n$-th moments of $Z$, with entropy regularization to promote ``equal'' distributions.
    To formalize the problem, we consider the following constrained optimization problem: 
    \begin{align*}
        P(\lambda): \quad \min_{x \in \bbR^{p+1}} \ \  & \frac12 \vvt{Ax-b}^2 + \lambda \cdot \sum_{j=1}^n x_{(j)} \log(x_{(j)}) \\ 
        \tn{s.t.} \ \  & {\bm 1}_{p+1}^T x = 1, \quad x \ge 0, 
    \end{align*}
    where $x_{(j)}$ is the $j$-th component of $x$, $A \in \bbR^{n \times (p+1)}$ with $A_{i,j} = w_j^i$, and $\lambda \in \Lambda = [10^{-2}, 10^2]$. 
    The parametric optimization problem $P(\lambda)$ is a constrained optimization problem, including an affine constraint, which does not satisfy our assumptions. 
    Therefore, we introduce a new variable $y$ to substitute $x$. 
    Let $y \in \bbR^p$ with $y_{(i)} = x_{(i)}$, for $i = 1, \dots, p$, and $S' = \{y \in \bbR^p: y \ge 0, {\bm 1}_p^T y \le 1\}$.
    Then the above moment matching problem $P(\lambda)$ is equivalent to the following parametric optimization problem: 
    \begin{align*}
        P'(\lambda): \quad \min_{y \in \bbR^p} \ \  & \frac12 \vvta{A' y - b'}^2 + \lambda \cdot \lr{\sum_{j=1}^n y_{(j)} \log(y_{(j)}) + (1 - {\bm 1}_p^T y) \log(1 - {\bm 1}_p^T y)} \\ 
        \tn{s.t.} \ \  & {\bm 1}_p^T y \le 1, \quad y \ge 0, 
    \end{align*}
    where $A' = A_{1:p} - A_{p+1} {\bm 1}_p^T$ and $b' = b - A_{p+1}$. 
    We examine the empirical behavior of each of the previously presented methods on $(P')$ without constraints, that is, $f(y) = \frac12 \vvta{A' y - b'}^2$ and $\Omega(y) = \sum_{j=1}^n y_{(j)} \log(y_{(j)}) + (1 - {\bm 1}_p^T y) \log (1 - {\bm 1}_p^T y)$. 
    \newtext{Translating $y$ back to $x$, we can show that the exact path $x(\lambda)$ for $\lambda \in [\lambda_{\min}, \lambda_{\max}]$ is a subset of the relative interior of $S$, but the proof is omitted for brevity.}
    Also, when the step-size $h$ is small enough, all grid points $\{x_k\}$ will be in the relative interior of $S$, and therefore, the approximate path $\hat{x}(\lambda)$ for $\lambda \in [\lambda_{\min}, \lambda_{\max}]$ is a subset of the relative interior of $S$. 
    
    \paragraph{Synthetic Data Generation Process}
    We generate the data $(x, w)$ according to the following generative model. 
    The true distribution vector $x^{\tn{true}}$ is according to the model $x_{(i)}^{\tn{true}} = \frac{\exp(z_{(i)})}{\sum_{j=1}^{p+1} \exp(z_{(j)})}$ for $i = 1, \dots, p+1$, where $z_{(i)} \sim \tn{unif}(0, 1)$. 
    The sample vector $w$ is generated from an independent uniform distribution, \ie $w_{(i)} \sim \tn{unif}(0, 1)$ for $i = 1, \dots, p$, and without loss of generality, we set $w_{(p+1)} = 0$.

	We examine the aforementioned exact and SOCG methods with different problem dimensions. 
	In the following set of experiments, we set the number of observations $n = 20$, the desired accuracy $\epsilon = 10^{-5}$, and vary the problem dimension $p \in \{128, 256, 512, 1024, 2048\}$. 
	The true distribution $x^{\text{true}}$ and the sample vector $w$ are synthetically generated according to the same process as before. 
	\Cref{fig:moment-matching-size} displays our result for this experiment. 
	Generally, we observe that the CPU time of computing an $\epsilon$-path increases as the dimension of the problem increases. 
	Comparing the CPU time of the exact methods and the SOCG methods, we notice that the SOCG methods are less sensitive to the size of the problems. 
	Again, the results demonstrate the superiority of SOCG methods for solving large-scale problems.

	\begin{figure}[tbhp]
		\centering
		\includegraphics[width=\textwidth]{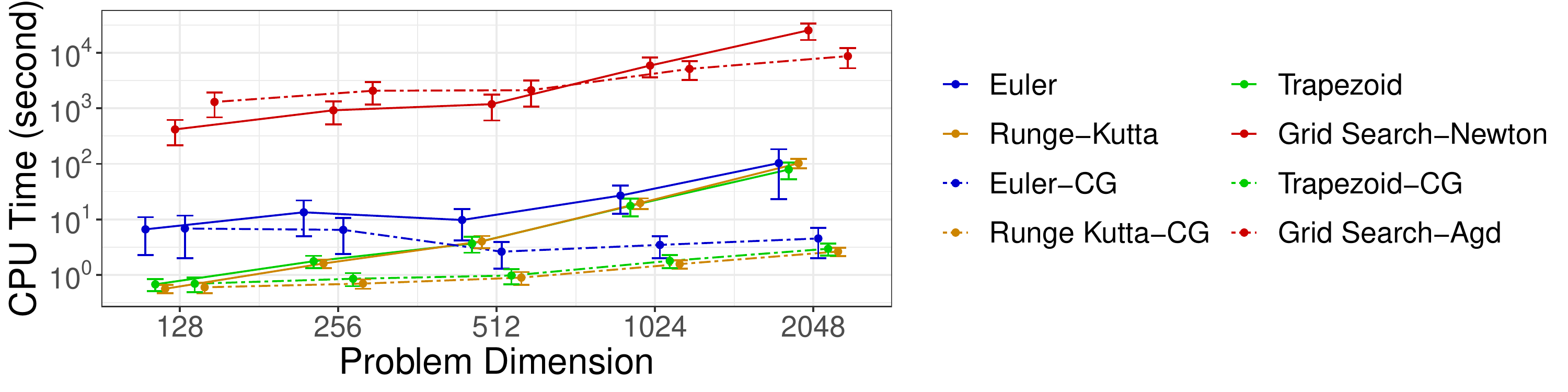}
        \caption{Exact and SOCG methods on moment matching problem with $n = 20$ and $\epsilon = 10^{-5}$, the CPU time comparison with different problem dimension. }
		\label{fig:moment-matching-size}
	\end{figure}

\section{Conclusion}
Inspired by a differential equation perspective on the parametric solution path, we developed and analyzed several different exact and inexact second-order algorithms. We develop theoretical results that are an improvement over the best-known results of grid search methods, and we present computational results that demonstrate the effectiveness of our methods.

\section*{Acknowledgments}
This research is supported, in part, by NSF AI Institute for Advances
in Optimization Award 2112533.

\bibliographystyle{plainnat}
{\small
\bibliography{ref}}

\appendix

\section{Proofs}

    \subsection{Proof of \cref{prop:ode}}
    For fixed $\lambda \in [\lambda_{\min}, \lambda_{\max}]$, we have $\frac{\xi(\lambda)}{\lambda}$ is a constant and its absolute value is no larger than $C$. 
    Let $H_i = \nabla^2 f(x_1) + \lambda \nabla^2 \Omega(x_1)$ and $g_i = \nabla f(x_i)$ for $i = 1, 2$. 
    Since $f(\cdot)$ is $\mu$-strongly convex and $\Omega(\cdot)$ is $\sigma$ strong convex, we have $\vvta{H_i} \ge \mu + \lambda_{\min} \sigma$ for $i = 1, 2$. 
    Since $f(\cdot)$ is $L$-Lipschitz continuous, we have $\vvta{g_2} \le L$. 
    Also, let $v_1 = H_1^{-1} g_1$, $v_2 = H_2^{-1} g_2$, and $v_1' = H_1^{-1} g_2$. 
    t holds that 
    \begin{equation*}
        \vvta{v_1 - v_1'} = \vvta{H_1^{-1} (g_1-g_2)} \le \vvta{H_1^{-1}} \cdot \vvta{g_1-g_2} \le \tfrac{1}{\mu+\lambda_{\min}\sigma} \cdot L \vvta{x_1-x_2}. 
    \end{equation*}
    Also, since $\vvta{H_1 - H_2} \le \vvta{\nabla^2 f(x_1) - \nabla^2 f(x_2)} + \lambda \vvta{\nabla^2 \Omega(x_1)- \nabla^2 \Omega(x_2)} \le L(1+\lambda_{\max}) \vvta{x_1-x_2}$ and $\vvta{H_1-H_2} = \vvta{H_1 (H_1^{-1}-H_2^{-1}) H_2} \le \vvta{H_1} \cdot \vvta{H_1^{-1}-H_2^{-1}} \cdot \vvta{H_2}$, it holds that $\vvta{H_1^{-1}-H_2^{-1}} \le \tfrac{L(1+\lambda_{\max}) \vvta{x_1-x_2}}{(\mu+\lambda_{\min}\sigma)^2}$. 
    Therefore, we have 
    \begin{equation*}
        \vvta{v_1'-v_2} = \vvta{\lr{H_1^{-1}-H_2^{-1}} g_2} \le \vvta{H_1^{-1}-H_2^{-1}} \cdot \vvta{g_2} \le \tfrac{L^2(1+\lambda_{\max}) \vvta{x_1-x_2}}{(\mu+\lambda_{\min}\sigma)^2}. 
    \end{equation*}
    To conclude, since $v(x_1, \lambda) - v(x_2, \lambda) = \frac{\xi(\lambda)}{\lambda} \cdot (v_1-v_2)$, it holds that 
    \begin{equation*}
        \vvta{v(x_1, \lambda) - v(x_2, \lambda)} \le (\tfrac{LC}{\mu+\lambda_{\min}\sigma} + \tfrac{L^2C(1+\lambda_{\max})}{(\mu+\lambda_{\min}\sigma)^2}) \cdot \vvta{x_1-x_2}. 
    \end{equation*}
    
    \subsection{Proof of \cref{lmm:trapezoid-local}}
	In the following residual analysis, we will work with the high-order directional derivatives of $f(\cdot)$ and $\Omega(\cdot)$. 
	We now introduce the definition of the directional derivatives. 
	For $p \ge 1$, let $D^p f(x) [h_1, \dots, h_p]$ denote the directional derivative of function $f$ at $x$ along directions $h_i, i = 1, \dots, p$. 
	For example, $D f(x) [h] = \nabla f(x)^T h$ and $D^2 f(x) [h_1, h_2] = h_1^T \nabla^2 f(x) h_2$. 
	Furthermore, the norm of directional derivatives is defined as $\vvta{D^p f(x)} := \max_{h_1, \dots, h_p} \lrrr{\vt{D^p f(x) [h_1, \dots, h_p]}: \vvta{h_i} \le 1}$.  
	For detailed properties of directional derivatives, we refer the reader to \citet{nesterov2018implementable}. 
	We start with computational guarantees of the Taylor approximation on functions with Lipschitz continuous high-order derivatives. 
	The following lemma guarantees the accuracy of the Taylor expansion. 
	\begin{lemma}[$(1.5)$ and $(1.6)$ in \citet{nesterov2018implementable}] \label{lm:acc-tay}
Let the function $\phi(\cdot)$ be convex and $p$-times differentiable. 
	    Suppose $p$-th order derivative of $\phi(\cdot)$ is $L_p$ Lipschitz continuous. 
	    Let $\Phi_{x, p}(\cdot)$ denote the Taylor approximation of function $\phi(\cdot)$ at $x$: $\Phi_{x, p}(y) := f(x) + \sum_{i=1}^p \frac{1}{i!} D^i \phi(x) [y-x]^i$. 
		Then we have the following guarantees: 
		\begin{equation*}
		    \vt{\phi(y) - \Phi_{x, p}(y)} \le \frac{L_p}{(p+1)!} \vvta{y-x}^{p+1}, \quad \vvta{\nabla \phi(y) - \nabla \Phi_{x, p}(y)} \le \frac{L_p}{p!} \vvta{y-x}^{p}. 
		\end{equation*}
	\end{lemma}
	\begin{condition}\label{cond:smallstep}
		Suppose step-size $h$ satisfies that $h \le \min \lrrr{0.2, \frac{\tilde{\mu}}{8L(1+G)}}$. 
	\end{condition}

    For simplicity, we use $(\tilde{x}, \tilde{\lambda})$ to represent $(x_{\tn{next}}, \lambda_{\tn{next}})$. 
	We will begin the complexity analysis of trapezoid method by the following two lemmas which provide proper upper bounds the norm of direction $\vvta{d_1}$ and the difference between $d_1$ and $d_2$. 

	\begin{lemma}\label{lm:bound-norm}
		Suppose $\sigma \ge 1$, $x \in S_{x_0}$, $\lambda \in [\lambda_{\min}, \lambda_{\max}]$, $h > 0$ and $(\tilde{x}, \tilde{\lambda}) = T(x, \lambda; h)$. 
		Let $r$ denote the initial residual $\vvta{\nabla f(x) + \lambda \nabla \Omega(x)}$ that satisfies $r \le \tilde{\mu}$. 
		Then we have $(1 + \lambda) \vvta{d_1} \le 2(G+1)$. 
	\end{lemma}
	\begin{proof}
		Let $H = \nabla^2 f(x) + \lambda \nabla^2 \Omega(x)$. 
		By definition $d_1 = - H^{-1} \nabla f(x)$.  
		When $\lambda \ge 1$, we have $(1+\lambda) \vvta{d_1} \le \frac{1+\lambda}{\lambda} G \le 2$. 
		Also when $\lambda \le 1$, it holds that 
		\begin{equation*}
			(1+\lambda) \vvta{d_1} \le (1+\lambda) \lr{\vvta{H^{-1} \lambda \nabla \Omega(x)} + \vvta{H^{-1} (\nabla f(x) + \lambda \nabla \Omega(x))}} \le 2(G+1). 
		\end{equation*}
	\end{proof}

	\begin{lemma}\label{lm:diff-ind-reg}
		Suppose \cref{as:third-lips} and \cref{cond:smallstep} hold. 
		Let $(x, \lambda) \in \bbR^n \times \bbR^+$, and let $x_i, \lambda_i, d_i, i \in \{1, 2\}$ be generated by the trapezoid update scheme defined in \cref{eq:trapezoid}, we have 
		\begin{equation*}
			\vvta{\tilde{H}_1 (d_2 - d_1) - \nabla^2 f(x_1) (x_1-x_2) - (\tilde{H}_1 - \tilde{H}_2) d_2} \le \tfrac{L}{2} \vvta{x_1-x_2}^2,  
		\end{equation*}
		where $\tilde{H}_i = \nabla^2 f(x_i) + \lambda_i \nabla^2 \Omega(x_i), i \in \lrrr{1, 2}$. 
		Furthermore, it holds that 
		\begin{equation*}
			\vvta{d_2-d_1} \le 2 h L \tau (\vvta{d_1} + \vvta{d_1}^2).  
		\end{equation*}
	\end{lemma}
	\begin{proof}
		Using the definition of $d_1, d_2$ we have 
		\begin{equation}\label{eq:diff-ind-reg-1}\begin{aligned}
			\tilde{H}_1 (d_2 - d_1) &= \nabla f(x_1) - \nabla f(x_2) + (I - \tilde{H}_1 \tilde{H}_2^{-1}) \nabla f(x_2) \\
			&= \nabla^2 f(x_1) (x_1 - x_2) + (\tilde{H}_1 - \tilde{H}_2) d_2 + (R), \\
		\end{aligned}\end{equation}
		where $\vvta{(R)} = \vvta{\nabla f(x_1) - \nabla f(x_2) - \nabla^2 f(x_1) (x_1 - x_2)} \le \tfrac{h^2}{2} \cdot L \vvta{d_1}^2$. 
		Also, we have $\vvta{\nabla^2 f(x_1) (x_1 - x_2)} = h \vvta{\nabla^2 f(x_1) d_1} \le h L \vvta{d_1}$, and that 
		\begin{equation*}\begin{aligned}
			& \vvta{(\tilde{H}_1 - \tilde{H}_2) d_2} \\
			= & \vvta{\lr{\nabla^2 f(x_1) - \nabla^2 f(x_2) + \lambda_1 (\nabla^2 \Omega(x_1) - \nabla^2 \Omega(x_2)) + (\lambda_1-\lambda_2) \nabla^2 \Omega(x_2)} d_2 } \\
			\le & \lr{L \vvta{x_1-x_2} + \lambda_1 L \vvta{x_1-x_2} + \vt{\lambda_1-\lambda_2} L} \vvta{d_2} 
			\le h \lr{L(\lambda+1) \vvta{d_1} + \lambda L} \vvta{d_2}. 
		\end{aligned}\end{equation*}
		Hence, it holds that 
		\begin{equation}\label{eq:lm:diff-eq2}\begin{aligned}
			&\tilde{\mu} \vvta{d_2} - \tilde{\mu} \vvta{d_1} \le \tilde{\mu} \vvta{d_2-d_1} \le \vvta{\tilde{H}_1 (d_2 - d_1)} \\
			\le & \tfrac{h^2}{2} \cdot L \vvta{d_1}^2 + h L \vvta{d_1} + h \lr{L(\lambda+1) \vvta{d_1} + \lambda L} \vvta{d_2}. 
		\end{aligned}\end{equation}
		When $h$ satisfies \cref{cond:smallstep}, it holds that $h \lr{L(\lambda+1) \vvta{d_1} + \lambda L} \le \tfrac{\tilde{\mu}}{3}$ and $\tfrac{h^2}{2} \cdot L \vvta{d_1} + h L \le \tfrac{\tilde{\mu}}{3}$. 
		Apply them to \cref{eq:lm:diff-eq2}, we have $\tfrac23 \tilde{\mu} \vvta{d_2} \le \tfrac43 \tilde{\mu} \vvta{d_1}$ and it implies that $\vvta{d_2} \le 2 \vvta{d_1}$. 
		Apply it to \cref{eq:lm:diff-eq2}, it holds that 
		\begin{align*}
			\vvta{\tilde{H}_1 (d_1 - d_2)} &\le \tfrac{h^2}{2} \cdot L \vvta{d_1}^2 + h L \vvta{d_1} + h \lr{L(\lambda_0+1) \vvta{d_1} + \lambda L} \vvta{d_2} \\
			&\le 2 h L (1+\lambda) (\vvta{d_1}+\vvta{d_1}^2).  
		\end{align*}
		Hence, it holds that $\vvta{d_1-d_2} \le h \cdot \tfrac{2 L (1+\lambda) (\vvta{d_1}+\vvta{d_1}^2)}{\mu + \lambda \sigma} \le 2 h L \tau (\vvta{d_1} + \vvta{d_1}^2)$. 
	\end{proof}

	Based on the results of \cref{lm:bound-norm} and \cref{lm:diff-ind-reg}, the following theorem analyzes the one-step residual accumulation of the trapezoid update in \cref{eq:trapezoid}. 
	\begin{proof}[Proof of \cref{lmm:trapezoid-local}]
		We will begin the local residual analysis by estimating the difference between $\nabla F_{\tilde{\lambda}}(\tilde{x})$ and $\tfrac{\tilde{\lambda}}{\lambda} \cdot \nabla F_{\lambda}(x)$. 
		For simplicity, let $x$, $\tilde{x}$, $\lambda$, and $\tilde{\lambda}$ denote $x_k$, $x_{k+1}$, $\lambda_k$, and $\lambda_{k+1}$ in the trapezoid update \cref{eq:trapezoid} respectively. 
		Then we have
		\begin{equation*}
			R = \nabla F_{\tilde{\lambda}}(\tilde{x}) - \tfrac{\tilde{\lambda}}{\lambda} \cdot \nabla F_{\lambda}(x) = \underbrace{\nabla f(\tilde{x}) - \nabla f(x)}_{(A)} + \underbrace{\tilde{\lambda} (\nabla \Omega(\tilde{x}) - \nabla \Omega(x))}_{(B)} + \underbrace{(1 - \tfrac{\tilde{\lambda}}{\lambda}) \nabla f(x)}_{(C)}. 
		\end{equation*}
		We will approach the result in \cref{eq:third-err} by splitting and rearranging the terms in $(A)$, $(B)$ and $(C)$. 
		From \cref{lm:acc-tay}, it holds that 
		\begin{equation*}\vvta{(RA)} := \vvta{ (A) - \underbrace{\nabla^2 f(x) (\tilde{x}-x)}_{(A')} - \underbrace{\tfrac12 D^3 f(x) [\tilde{x}-x]^2}_{(A3)} } \le \tfrac{L}{6} \vvta{\tilde{x}-x}^3 = \tfrac{h^3}{6} L \Vert \tilde{d} \Vert^3. \end{equation*}
		From the update \cref{eq:trapezoid}, we have $(A') = \underbrace{h \nabla^2 f(x) d_1}_{(A1)} + \underbrace{\tfrac{h}{2} \nabla^2 f(x) (d_2-d_1)}_{(A2)}$. 
		For $(B)$, using \cref{lm:acc-tay} and based on update \cref{eq:trapezoid} we have 
		\begin{equation*}\begin{aligned}
			(B) &= \lambda \nabla^2 \Omega(x) (\tilde{x} - x) + \underbrace{(\tilde{\lambda}-\lambda) \nabla^2 \Omega(x) (\tilde{x} - x)}_{(B3)} + \underbrace{\tilde{\lambda} \cdot \tfrac12 D^3 \Omega(x) [\tilde{x}-x]^2}_{(B4)} + (RB) \\
			&= \underbrace{h \lambda \nabla^2 \Omega(x) d_1}_{(B1)} + \underbrace{\tfrac{h}{2} \lambda \nabla^2 \Omega(x) (d_2-d_1)}_{(B2)} + (B3) + (B4) + (RB),   
		\end{aligned}\end{equation*}
		where $\vvta{(RB)} = \tilde{\lambda} \vvta{\nabla \Omega(\tilde{x}) - \nabla \Omega(x) - \nabla^2 \Omega(x) (\tilde{x} - x) - \tfrac12 D^3 \Omega(x) [\tilde{x}-x]^2} \le \tfrac{h^3 \tilde{\lambda} L}{6} \vvta{\tilde{d}}^3$. 
		Also, $(C) = (h - \tfrac{h^2}{2}) \nabla f(x) = \underbrace{h \nabla f(x)}_{(C1)} - \underbrace{\tfrac{h^2}{2} \nabla f(x)}_{(C2)}$. 
		By the definition of $d_1$, we have 
		\begin{equation}\label{eq:local-trap-1}
			(A1) + (B1) + (C1) = h \nabla^2 f(x) d_1 + h \lambda \nabla^2 \Omega(x) d_1 + h \nabla f(x) = 0. 
		\end{equation}
		Using \cref{lm:diff-ind-reg}, we have
		\begin{equation}\label{eq:local-trap-2}\begin{aligned}
			(A2) + (B2) &= \underbrace{\tfrac{h}{2} \nabla^2 f(x) (x_1-x_2)}_{(D1)} + \underbrace{\tfrac{h}{2} (\nabla^2 f(x_1) - \nabla^2 f(x_2)) d_2}_{(D2)} \\ 
			&\phantom{=} + \underbrace{\tfrac{h}{2} (\lambda_1  \nabla^2 \Omega(x_1) - \lambda_2 \nabla^2 \Omega(x_2)) d_2}_{(D3)} + (RD), 
		\end{aligned}\end{equation}
		where $\vvta{(RD)} \le \tfrac{h}{2} \cdot \tfrac{L}{2} \vvta{x_1-x_2}^2 = \tfrac{h^3}{4} L \vvta{d_1}^2$. 
		Furthermore, it holds that 
		\begin{equation}\label{eq:local-trap-3}
			(D1) + (C2) = -\tfrac{h^2}{2} \nabla^2 f(x) d_1 - \tfrac{h^2}{2} \nabla f(x) = \underbrace{\tfrac{h^2}{2} \lambda \nabla^2 \Omega(x) d_1}_{(E1)}, 
		\end{equation}
		and 
		\begin{equation*}\begin{aligned}
			(A3) + (D2) &= \tfrac{h^2}{2} D^3 f(x) \lrr{\tfrac12 (d_1+d_2)}^2 + \tfrac{h}{2} D^3 f(x) [x_1-x_2, d_2] + (R1) \\
			&= \underbrace{\tfrac{h^2}{2} D^3 f(x) \lrr{\tfrac12 (d_1-d_2)}^2}_{(R2)} + (R1), \\
		\end{aligned}\end{equation*}
		where 
		\begin{equation*}\begin{aligned}
			\vvta{(R1)} &= \vvta{\tfrac{h}{2} \lr{\nabla^2 f(x_1) - \nabla^2 f(x_2)} d_2 - \tfrac{h}{2} D^3 f(x) [x_1-x_2, d_2]} \\
			&\le \tfrac{h}{2} \cdot \tfrac{L}{2} \vvta{x_1-x_2}^2 \vvta{d_2} = \tfrac{h^3}{4} L \vvta{d_1}^2 \vvta{d_2}. 
		\end{aligned}\end{equation*}
		We further have $(D3) = \underbrace{\tfrac{h}{2} \lambda_2 \lr{\nabla^2 \Omega(x_1) - \nabla^2 \Omega(x_2)} d_2}_{(E2)} + \underbrace{\tfrac{h}{2} (\lambda_1 - \lambda_2) \nabla^2 \Omega(x_1) d_2}_{(E3)}$, 
		and 
		\begin{equation*}\begin{aligned}
			(B4) + (E2) &= \tilde{\lambda} \cdot \tfrac12 D^3 \Omega(x) [\tilde{x}-x]^2 + \tfrac{h}{2} \lambda_2 D^3 \Omega(x_1) [x_1-x_2, d_2] + (R5) \\
			&= \underbrace{\tfrac{h^2}{2} (\tilde{\lambda}-\lambda_2) D^3 \Omega(x) \lrr{\tfrac{d_1+d_2}{2}}^2}_{(R3)} + \underbrace{\tfrac{h^2}{2} \lambda_2 D^3 \Omega(x_1) \lrr{\tfrac{d_1-d_2}{2}}^2}_{(R4)} + (R5), \\
		\end{aligned}\end{equation*}
		where 
		\begin{equation*}\begin{aligned}
			\vvta{(R5)} &= \vvta{\tfrac{h}{2} \lambda_2 \lr{\nabla^2 \Omega(x_1) - \nabla^2 \Omega(x_2)} d_2 - \tfrac{h}{2} \lambda_2 D^3 \Omega(x) [x_1-x_2, d_2]} \\
			&\le \tfrac{h}{2} \lambda_2 \cdot \tfrac{L}{2} \vvta{x_1-x_2}^2 \vvta{d_2} = \tfrac{h^3}{4} \lambda_2 L \vvta{d_1}^2 \vvta{d_2}, 
		\end{aligned}\end{equation*}
		and 
		\begin{equation*}\begin{aligned}
			(B3) + (E1) + (E3) = & (-h+\tfrac{h^2}{2}) \lambda \nabla^2 \Omega(x) h \cdot \tfrac{d_1+d_2}{2} + \tfrac{h^2}{2} \lambda \nabla^2 \Omega(x) d_1 \\
			& + \tfrac{h}{2} \lr{h - h^2} \nabla^2 \Omega(x) d_2 = \underbrace{\tfrac{h^3}{4} \lambda \nabla^2 \Omega(x) (d_1-d_2)}_{(R6)}. 
		\end{aligned}\end{equation*}
		Hence, it holds that 
		\begin{equation}\label{eq:loc-1}\begin{aligned}
			\vvta{(R)} \le & \vvta{(RA)} + \vvta{(RB)} + \vvta{(RD)} + \vvta{(R1)} + \vvta{(R2)} \\ 
			& + \vvta{(R3)} + \vvta{(R4)} + \vvta{(R5)} + \vvta{(R6)} \\
			\le & h^3 \cdot \tfrac{L}{6} \vvta{\tfrac{d_1+d_2}{2}}^3 + h^3 \cdot \tfrac{\tilde{\lambda} L}{6} \vvta{\tfrac{d_1+d_2}{2}}^3 + h^3 \cdot \tfrac{L}{6} \vvta{d_1}^2 \\ 
			& + h^3 \cdot \tfrac{L}{4} \vvta{d_1}^2 \vvta{d_2} + h^2 \cdot \tfrac{L}{8} \vvta{d_1-d_2}^2 + h^4 \cdot \tfrac{\lambda L}{4} \vvta{\tfrac{d_1+d_2}{2}}^2 \\ 
			& + h^2 \cdot \tfrac{\lambda_2 L}{8} \vvta{d_1-d_2}^2 + h^3 \cdot \tfrac{\lambda_2 L}{4} \vvta{d_1}^2 \vvta{d_2} + h^3 \cdot \tfrac{\lambda L}{4} \vvta{d_1-d_2} \\ 
			\le & h^3 \cdot \tfrac{L}{3} \vvta{d_1}^3 + h^3 \cdot \tfrac{L \lambda}{3} \vvta{d_1}^3 + h^3 \cdot \tfrac{L}{6} \vvta{d_1}^2 + h^3 \cdot \tfrac{L}{2} \vvta{d_1}^3 \\
			& + h^2 \cdot \tfrac{L}{8} \vvta{d_1-d_2}^2 + h^4 \cdot \tfrac{\lambda L}{2} \vvta{d_1}^2 + h^2 \cdot \tfrac{\lambda L}{8} \vvta{d_1-d_2}^2 \\ 
			& + h^3 \cdot \tfrac{\lambda L}{2} \vvta{d_1}^3 + h^3 \cdot \tfrac{\lambda L}{4} \vvta{d_1-d_2} \\
			\le & h^3 L (1+\lambda) (\vvta{d_1}^3 + \vvta{d_1}^2 + \vvta{d_1}) + h^2 \cdot \tfrac{L(1+\lambda)}{8} \vvta{d_1-d_2}^2. 
		\end{aligned}\end{equation}
		Apply the result in \cref{lm:diff-ind-reg} into \cref{eq:loc-1}, we can further get 
		\begin{equation}\label{eq:loc-2}
			\vvta{d_1-d_2}^2 \le (2 h L \tau (\vvta{d_1} + \vvta{d_1}^2))^2 \le 8 h^2 L^2 \tau^2 (\vvta{d_1}^2 + \vvta{d_1}^4). 
		\end{equation}
		Also, by applying the results in \cref{lm:bound-norm} to \cref{eq:loc-1,eq:loc-2}, we have 
		\begin{equation*}\begin{aligned}
			\vvta{(R)} \le h^3 \cdot 3 L (1+G)^3 + h^4 \cdot 2 L^3 \tau^2 (1+G)^4.
		\end{aligned}\end{equation*}
	\end{proof}

\end{document}